\documentclass[11pt]{article}
\usepackage[a4paper,margin=2cm]{geometry}

\RequirePackage{fix-cm}

\usepackage{amsthm}
\usepackage{color,xcolor}
\usepackage{amsmath,amssymb,bm}
\usepackage{graphicx,graphics}
\definecolor{refgreen}{rgb}{0,0.5,0}
\usepackage[utf8]{inputenc}
\usepackage{textcomp} 
\usepackage{mathtools}
\usepackage[numbers]{natbib}

\def   \d {\hspace{1.5pt}\mathrm{d}}

\usepackage{hyperref}
\usepackage{ulem} 

\usepackage{cleveref}

\usepackage{algorithm} 

\usepackage{diffcoeff}
\usepackage{bm} 

\definecolor{refblue}{rgb}{0,0,0.75}
\definecolor{refblueb}{rgb}{0,0,1}
\definecolor{refgreen}{rgb}{0.13,0.55,0.13}
\definecolor{refred}{rgb}{1,0,0}

\hypersetup{
	colorlinks   = true, 
	urlcolor     = refblue, 
	linkcolor    = refblueb, 
	citecolor   = refgreen 
}

\usepackage{stmaryrd} 

\usepackage{longtable}

\newcommand{\R}{\mathbb{R}}
\newcommand{\NaGa}{\nabla_\Gamma}

\DeclareMathOperator{\osc}{osc}
\DeclareMathOperator{\TOL}{TOL}

\newcommand{\IntRef}{\mathcal{I}_{\text{ref}}^n} 

\newcommand{\Rcoarse}{\mathcal{R}_{\textnormal{c}}}
\newcommand{\Ga}{\varGamma}
\newcommand{\baruhtaulift}{\overline{\vphantom{\bar u}u_{h,\tau}}}
\newcommand{\baruhtau}{\overline{\vphantom{\bar U}U_{h,\tau}}}
\newcommand{\comesh}{n-1 \oplus n} 

\newcommand{\partialm}{\partial^\bullet} 
\newcommand{\partialmh}{\partial_h^\bullet} 

\usepackage{graphicx}
\usepackage{epstopdf}

\usepackage[autostyle=true]{csquotes}

\graphicspath{ {./Figures/} }
\usepackage{graphicx}
\usepackage{subcaption}
\usepackage{todonotes}

\usepackage{pgfplots}

\newtheorem{theorem}{Theorem}
\newtheorem{proposition}[theorem]{Proposition}%
\newtheorem{remark}{Remark}%

\numberwithin{equation}{section}

\renewcommand{\nu}{\textnormal{n}}

\makeatletter
\renewcommand{\d}{\mathop{}\!\mathrm{d}}
\makeatother

\begin{document}

\title{A posteriori error estimates for parabolic PDEs on evolving surfaces}

\author{
Michael Lantelme\\
Institute of Mathematics\\
Paderborn University\\
Warburger Str.\ 100\\
33098 Paderborn, Germany\\[1ex]
\texttt{lantelme@math.uni-paderborn.de}
}

\date{\today}

\maketitle

\begin{abstract}
We derive residual-based a posteriori error estimates for parabolic surface PDEs on closed evolving surfaces. The main contribution is to prove efficiency and reliability for the proposed error indicator, which bounds the error quantities globally from above and globally in space and locally in time from below. We extend methods from \cite{KL25} to allow for non-trivial coarsening on evolving surfaces. Multiple numerical experiments are given, which illustrate the asymptotic behaviour of the error and effectiveness of the refinement and coarsening.
\end{abstract}


\vspace{1em}

\section{Introduction}
This paper covers residual-based a posteriori error estimates for parabolic partial differential equations (PDEs) on closed evolving surfaces. 
In particular efficiency and reliability between the errors and estimators.
This includes that we show that the derived error indicators bound the error globally in space and time from above, up to oscillation. And globally in space and locally in time from below, up to oscillation, high-order geometric, mesh-transfer, and movement related indicators. 
Most of the analysis is done for a model problem, the heat equation on evolving surfaces, however the results are easily extendible to general parabolic surface PDEs.
The discrete approximation is done via evolving linear surface finite element method (ESFEM) in space, where the vertices move along the exact flow, and backwards Euler discretization in time.
We provide a minimal working example for a space--time adaptive algorithm for which a set of numerical tests is done to accompany the main results. 
Additionally we expect the tools, in particular related to coarsening, to generalize to settings with evolving computational domains whose nodes do not coincide with the continuous geometry.

The a posteriori analysis of parabolic PDEs on evolving surfaces poses substantial new challenges compared to the stationary case, due to the interplay between temporal evolution, surface transport, and changing geometry. Based on the insights gained from the stationary setting in \cite{KL25}, we develop new techniques and analytical tools to control the additional terms arising from the surface flow and geometry changes.
The a posteriori error analysis is done, similar to, e.g.\, \cite{Verfuerth2003,Kreuzer12,KL25}, by splitting the residual into different subresiduals. Namely this includes the typical spatial, temporal and oscillation residuals (standard for parabolic PDEs in euclidean domains \cite{Verfuerth2003,Kreuzer12}), the geometric residual (standard for elliptic surface PDEs \cite{Demlow2007,Camacho2014}), and the novel movement residual, the velocity induced temporal residual and some high-order geometric terms resulting from the velocity discretization.
As for the stationary setting it is essential to note that the error-analysis is complicated due to working with the non-conform ESFEM discretization, which effects both the approximation, and the velocity induced by the flow.
The evolution of the domain results in the need to compare functions on different domains due to the flow but also due to adaptivity.
Similar to \cite[Section~3.2]{KL25}, but complicated by the evolution of the domain, we introduce a refinement interpolation to define the full-discrete formulation on a singular mesh.
The refinement interpolant allows us to define two time interpolations of the discrete numerical solution. One of which is continous and can not be unlifted to the discrete surface easily, but is essential for the error analysis. The other is discontinous at the discrete time-levels, but easy to unlift and essential to derive computable indicators.
The indicators which demand novel techniques are:
The movement residual which handles the domain difference resulting from the flow, which was split such that we can measure how the quantities of the full discrete weak formulation behave under domain changes. 
The coarsening residual requires careful mesh generation and its bounds are split into a coarsening and explicit mesh-transfer part.
To define a typical coarsening indicator it is essential to establish a construction which avoids node mismatches under the flow due to time-dependent non-linear lifting combined with arbitrary flow of the surface.
 This results in the need of a careful construction of the refinement interpolation.
Additionally, the interpolation results for the common interpolation operator for surfaces introduced in \cite[Section~3.3]{KL25} are generalized. 

There are various applications for time dependent surface PDEs in multiple fields, including surfactant transport on moving interfaces \cite{BGN_surfactant}, biological growth and tumour modelling \cite{Eyles2019ATM,King2021FreeBP}, phase separation on deforming surfaces \cite{ElliottRanner,ElliottSales2024_CH}, pattern formation on evolving surfaces \cite{Amago}, and geometric flows, e.g.\, mean curvature flow \cite{Huisken1987TheVP,MCF,MCF_surgery}.

Evolving surface PDEs were first introduced by Dziuk and Elliott in \cite{Dziuk2007ESFEM}, based on the surface finite element method introduced by Dziuk \cite{Dziuk1988}.
A good overview about the ESFEM can be found in \cite{Dziuk2013FiniteEM, RannerElliott}. We refer to \cite{BonitoDemlowNochetto} for a general survey on different finite element methods for surfaces. 

Adaptive methods for elliptic PDEs on stationary surfaces were first analysed in \cite{Demlow2007}, the results were later extended \cite{Camacho2014} to derive efficient and reliable $L^2$ and pointwise indicators. The split into geometric and consistency contributions were key to determine efficient bounds. 
An algorithm which explicitly handles the resolution of the geometry to guarantee convergence of the method and shape-regularity of the meshes, which in comparison to the euclidean case does not follow immediately, was later developed in \cite{Bonito2013}.
A posteriori error analysis for parabolic surface PDEs was first developed in \cite{KL25} for stationary PDEs, where they extended and combined the theories for elliptic surface PDEs and parabolic PDEs on euclidean domains, in particular based on the frameworks of \cite{Verfuehrt1996,Verfuerth2003,chen2004feng_adap_flat,Kreuzer12}. 

We note that the present framework yields suboptimal $L^\infty (L^2)$-bounds. To improve this, techniques based on strong stability estimates \cite{eriksson95}, or elliptic reconstruction \cite{MakridakisNochetto2003,ell_reconstruct} appear promising, and our results may provide a useful basis for such an analysis.  

To our knowledge, a posteriori error analysis and adaptivity was not yet studied for \textit{parabolic problems on evolving surfaces} in the literature.

The paper is organized as follows: 
In Section~\ref{ch:surfacePDE}, we introduce the heat equation on evolving surfaces. 
Section~\ref{sec:discretization} recalls the ESFEM framework, derives the semi-discrete formulation, and introduces the full-discrete implicit Euler scheme based on a refinement interpolant. This construction yields two distinct time interpolants of the discrete solution. Additionally we develop a smallest common refinement for evolving surface meshes.
 The main results and the explicit error indicators are stated in Section~\ref{ch:main_result}. 
 Section~\ref{ch:proof} contains their proofs, including: equivalence between the residual and error, a residual decomposition into different subresiduals related to different effects of the PDE-error, and finally reliability bounds, together with efficiency estimates for selected contributions.
 Finally in Section~\ref{sec:numerical_experiments} we propose a space--time adaptive algorithm and numerical experiments.

\section{PDEs and Differential Operators on Surfaces}
\label{ch:surfacePDE}
\subsection{Preliminaries and Notations}
\label{sec:prelim}
We follow the setup and notational convention introduced in \cite{Dziuk2013FiniteEM}. 
Let us consider a closed two-dimensional sufficiently smooth (atleast $C^3$) evolving hypersurface $\Gamma(t)$ . We assume that for each $t \in [0,T]$, the surface $\Gamma(t)$ is described as the zero-level-set of a signed distance function $d(\cdot,t) \colon \mathcal{U}_\epsilon(t) \subset \R ^{3} \rightarrow \R$, where $\mathcal{U}_\epsilon(t)$ is a tubular region around $\Gamma(t)$ with width $\epsilon(t) > 0$.
A scalar function $u(x,t)$ ($x \in \Gamma(t), 0 \leq t \leq T$) has the tangential gradient   
\begin{align*}
	\nabla_{\Gamma(t)} u := \nabla \overline{u} - (\nu \cdot \nabla \overline{u}) \nu = (I-\nu \otimes \nu) \nabla \overline{u} =: P \nabla \overline{u} .
\end{align*} 
Where $\nu = \nu(\cdot,t)$ is the outer normal vector field to $\Gamma(t)$, $\overline{u}$ denotes the extension of $u$ onto $\mathcal{U}_\epsilon(t)$ and $P$ is the tangential projection. The Laplace--Beltrami operator is given by $\Delta_{\Gamma(t)} u = \nabla_{\Gamma(t)} \cdot \nabla_{\Gamma(t)} u$.
We define the space--time manifold $Q_T := \bigcup_{t \in [0,T]} \Gamma(t) \times \{t\}$.
For sufficiently small $\epsilon(t)>0$, every $x \in \mathcal{U}_\epsilon(t)$ admits a unique closest point projection $y(x,t) \in \Gamma(t)$, see \cite{Demlow2007}. From now on assume the projection exists for all times with a $t$-independent $\epsilon$, similar to \cite[Section~5]{Dziuk2013FiniteEM}, and is given by
\begin{align}
	\label{eq:uniquedecom}
	x = y(x,t)+d(x,t) \nu (y(x,t),t) .
\end{align}
The evolution of the surface is assumed to be governed by the flow map $G(\cdot,t):\Gamma^0 \rightarrow \Gamma(t)$ with regularity $G \in C^1([0,T];C^3(\Gamma^0))$, where $\Gamma^0 = \Gamma(0)$ and such that the flow map $G$ is a diffeomorphism from $\Ga^0$ to $\Ga(t)$ for $t \in [0,T]$. 
With a slight abuse of notation, for $t,s \in [0,T]$ we define the flow of points $x \in \Gamma(t)$ from time $t$ to $s$ by writing $G(\cdot,t,s):\Ga(t) \rightarrow \Ga(s)$ defined by $G(x,t,s):=G((G(\cdot,t))^{-1}(x),s)$. 

The associated surface velocity $v(\cdot,t)$ is defined by
\begin{align}
	\label{eq:flow_ode}
	\frac{\partial}{\partial t} G(\cdot,t) = v(G(\cdot,t),t) \quad \forall t \in (0,T] \quad \text{and} \quad G(\cdot,0)=Id.
\end{align}
On evolving surfaces the notion of time derivatives is extended by the material derivative to account for the mass transport
\begin{align*}
	\partialm f = \frac{\partial f}{\partial t} + v \cdot \nabla f.
\end{align*}

We state two commonly used results required for the analysis and discretization via the surface finite elements method. 
The surface variant of Greens formula \cite[Theorem~2.14]{Dziuk2013FiniteEM} reads
\begin{align}
	\label{eq:green_formula}
	\int_\Gamma \NaGa f \cdot \NaGa g = -\int_\Gamma f \cdot \Delta_\Ga g,
\end{align} 
given $f \in H^1(\Ga)$ and $g \in H^2(\Gamma)$. 
Additionally, the Leibniz formula on an evolving surface $\Gamma(t)$, for sufficiently smooth $f$, reads
\begin{align}
	\label{eq:leibniz_formula}
	\frac{\d}{\d t} \int_{\Gamma(t)} f = \int_{\Gamma(t)} \partialm f+ f \nabla_{\Gamma(t)} \cdot v.
\end{align}

\subsection{Heat equation on evolving closed surfaces}
The strong formulation of the surface heat equation with an inhomogeneity $f \in L^2(Q_T)$ and initial condition $u^0 \in L^2(\Gamma^0)$ reads
\begin{equation}
	\label{eq:heat_strong}
	\begin{aligned}
	\partialm u+u \NaGa \cdot v - \Delta_\Gamma u &= f \qquad &&\text{on} \,\,\, \Gamma(t) \quad t \in(0,T], \\
	u(\cdot,0) &= u^0 \qquad &&\text{on} \,\,\, \Gamma^0.
	\end{aligned}
\end{equation}

Utilising Greens formula \eqref{eq:green_formula} we derive the weak formulation. We abbreviate $\int_{\Gamma(t)} fg =: (f,g)_{L^2(\Gamma(t))}$.
The weak formulation reads: find $u \in H^1(Q_T)$ with $u(\cdot,0) = u^0$ such that it satisfies
\begin{align}
	\label{eq:weak_form}
	(\partialm u,\phi)_{L^2(\Gamma(t))} + (\nabla_\Gamma u, \nabla_\Gamma \phi)_{L^2(\Gamma(t))} + (u,\phi \NaGa \cdot v)_{L^2(\Gamma(t))} = (f,\phi)_{L^2(\Gamma(t))},
\end{align}
for almost all $t \in (0,T)$, and $\phi(\cdot, t) \in H ^1 (\Gamma(t))$.

\section{Dicretization}
\label{sec:discretization}
We employ the ESFEM method to derive the semidiscrete formulation. To be able to compare discrete solutions to the exact solution of \eqref{eq:heat_strong} the lift is introduced. Using the implicit Euler method in time we derive a full discrete version of the problem in a suitable form for a posteriori error analysis, which gives rise to the refinement interpolation operator. Utilising the set of full discrete solutions we introduce two affine interpolations, one of which is continous, the other easy to lift. Afterwards bilinear  forms are defined and we state results for the later error analysis. Finally we introduce the smallest common refinement, which is necessary to define coarsening-type error indicators.

\subsection{Semidiscrete ESFEM}
\label{sec:semidiscreteESFEM}
Approximate the evolving surface $\Gamma(t)$ by an evolving discrete surface $\Gamma_h(t)$ such that its vertices $\{X_j(t)\}_{j=1} ^N$, lie on $\Ga(t)$ hence $\Ga_h^n(t)$ interpolates $\Gamma(t)$. The surface $\Gamma_h(t)$ is assumed to be an admissible triangulation for all times (see  \cite{Dziuk2013FiniteEM}), which is smooth in time and given as a finite union of triangles $\mathcal{T}_h (t)$.
The interpolative property requires that the vertices $X_j$ move along the flow $X_j(t) = G(X_j(0),t)$ for all $j = 1,\dots,N$.

The evolving finite element space is given by
\begin{align*}
	S_h(t) = \{W_h \in C^0(\Gamma_h(t)) \mid W_h|_T  \, \text{affine linear} \,\, \forall T \in \mathcal{T}_h (t) \}.
\end{align*} 
The nodal basis functions $\{\Phi_j(t)\}$ span the finite element space, i.e.\ $S_h = \operatorname{span}\{\{\Phi_j(\cdot,t)\}_{j = 1} ^N\}$ for all $t \in [0,T]$. Thus we can write for arbitrary $W_h(\cdot,t) \in S_h(t)$ with nodal values $W_j(t) := W_h(X_j(t),t)$
\begin{equation*}
	W_h(\cdot,t) = \sum_{j=1} ^N W_j(t) \Phi_j(\cdot,t).
\end{equation*} 
This allows us to define a discrete material velocity $V_h$ on $\Gamma_h(t)$ as the interpolation of $v$, and the elementwise discrete material derivative 
$\partialmh W_h$ by 
\begin{align*}
	V_h(\cdot,t) := \sum_{j=1}^N \frac{\d X_j(t)}{\d t}  \Phi_j(\cdot,t), \quad
	\partialmh W_h|_{T(t)} := \bigg(\frac{\partial W_h(\cdot,t)}{\partial t}+ V_h(\cdot,t) \cdot \nabla_{\Gamma_h(t)} W_h(\cdot,t)\bigg)\bigg|_{T(t)}.
\end{align*}
Note that the flow \eqref{eq:flow_ode} imposes $\frac{\d X_j(t)}{\d t} = v(X_j(t),t)$.
By construction the nodal basis functions fulfil the transport property $\partialmh \Phi_j = 0$ on $\Gamma_h(t)$ \cite[Proposition~5.4]{Dziuk2007ESFEM}. Thus $\partialmh{W_h}(\cdot,t) = \sum_{j=1} ^N \frac{\d{W}_j}{\d t} \Phi_j(\cdot,t)$.

We formulate the semidiscrete problem: Given $F_h$, an appropriate approximation of $f$ on the discrete surface, find $U_h(\cdot,t) \in S_h(t)$ with
\begin{align}
	\label{eq:semidiscreteheat}
	\left(\partialmh U_h,\Phi_h \right)_{L^2(\Gamma_h(t))} + (\nabla_{\Gamma_h(t)} U_h, &\nabla_{\Gamma_h(t)} \Phi_h)_{L^2(\Gamma_h(t))} \nonumber \\
	+ (U_h,\Phi_h \nabla_{\Gamma_h(t)} \cdot V_h)_{L^2(\Gamma_h(t))} &= (F_h,\Phi_h)_{L^2(\Gamma_h(t))} \qquad \forall \Phi_h \in S_h (t) .
\end{align}
Note that this is different from the formulation of \cite[Definition~5.6]{Dziuk2007ESFEM}. This is for good reason as we want to avoid working with $\frac{d}{dt} (U_h,\Phi_h)_{L^2(\Gamma_h(t))}$ when discretizing in time, which results in differences of functions at two distinct discrete timesteps, and additionally due to adaptivity, to distinct mesh-connectivities. It is in general non-trivial to determine typical consistency error indicators (see first term of \eqref{eq:indicator - spatial}) if the discrete quantities are defined at distinct times.

\textbf{Lift.} 
The later error analysis requires us to represent both continuous and discrete quantities on a shared domain, the lift operator allows us in particular to represent discrete quantities on the exact surface $\Ga(t)$. We employ the closest point projection \eqref{eq:uniquedecom}, which requires $\Gamma_h(t) \subset \mathcal{U}_\epsilon(t)$ and $C^2$-regularity of $\Gamma(t)$, to uniquely \textit{lift} points and functions  between $\Ga_h(t)$ and $\Ga(t)$. 
For $x \in \Gamma_h(t)$ its lift is the unique solution of \eqref{eq:uniquedecom}, denoted as $x^{\ell(t)} \in \Ga(t)$. Note that the lift introduces a bijective map from $\Ga_h(t)$ to $\Ga(t)$ for all $t$.

Consequently, the \textit{lift of a function} $W_h \colon \Ga_h(t) \rightarrow \R$ onto $\Ga$ is given by $W_h^{\ell(t)} (x^{\ell(t)}) := W_h (x)$. As the lift is bijective we define the \textit{unlift} $w^{-\ell(t)} \colon \Ga_h(t) \to \R$ such that $(w^{-\ell(t)})^{\ell(t)} = w \colon \Ga(t) \to \R$ holds.

The lift of functions enables us to also lift discrete functions $\Phi_h \in S_h(t)$ resulting in $\Phi_h^{\ell(t)} \in S_h^{\ell(t)}(t) := \operatorname{span} \{(\Phi_j^{\ell(t)})_{j=1}^{N}\}$, note that the elements of the lifted discrete mesh are curved triangles, whose union exactly resembles $\Gamma(t)$.

For readability, we often suppress the lift notation. Throughout the paper, quantities denoted by capital letters are understood to be defined on $\Ga_h(t)$, whereas the corresponding  lower-case quantities are understood to be defined on $\Ga(t)$. In particular, the same letter in upper- and lower-case implicitly indicates that the two quantities are related by a (un)lift. For example, we write $X := x^{\ell(t)}$ for vertices and $W := w^{\ell(t)}$ for functions. 

The standard norm equivalence under lifts \cite[Lemma~3]{Dziuk1988} holds for all fixed times and for any $w \in H^1(\Gamma(t))$,
\begin{equation}
\label{eq:norm_equiv}
	\begin{gathered}
		\frac{1}{c} \|W\|_{L^2(\Ga_h(t))} \leq \|w\|_{L^2 (\Ga(t))} \leq c \|W\|_{L^2(\Ga_h(t))} , \\
		\frac{1}{c} |W|_{H^1(\Ga_h(t))} \leq |w|_{H^1(\Ga(t))} \leq c |W|_{H^1(\Ga_h(t))} ,
	\end{gathered}
\end{equation}
where we used the typical seminorm notational convention $|w|_{H^1(\Ga(t))} := \|\nabla_\Ga w\|_{L^2(\Ga(t))}$.

\subsection{Full discretization and time interpolation}
\label{ch:subsec:full_dis_time_int}
Assume that the temporal domain is split into $K$ timesteps $0=t^0 < t^1 < \dotsb < t^K = T$ which build intervals $(t^{n-1}, t^n]$ of length $\tau^n$, such that $\sum_{j=1}^K \tau^j = t^n \leq T$. This temporal dependence will always be reflected by the superscript $^n$.
Corresponding to each timestep we write $\Gamma_h^n := \Gamma_h(t^n)$ for the discrete admissible triangulations, the respective finite element spaces $S_h^n:= \operatorname{span} \{\Phi_1^n, \dotsc, \Phi_{N^n}^n\}$ with basis functions $\Phi_j^n := \Phi_j(\cdot,t^n)$, and the possible timestep dependent degrees of freedom $N^n$.
The temporal superscript is also used for the lift $^{\ell^n}: \Ga_h^n \to \Ga(t^n)$. 
 
Due to the movement of the surface it is immediately clear that $\Ga_h^{n-1} \neq \Ga_h^{n}$ in general. However, composed with the flow map applied to the vertices of the discrete meshes, it is possible to construct methods with mesh alignment under movement. On the other hand in an adaptive setting, due to refinement and coarsening, the discrete surfaces change non-trivially in each timestep. Thus even flowing discrete meshes along the exact flow will not guarantee that consecutive meshes will align. This is the main concern when dealing with time discretization.

We employ backwards difference method to derive the full discretization of \eqref{eq:semidiscreteheat}. To be able to compare $W_h^{n-1} \in S_h^{n-1}$ and $W_h^{n} \in S_h^{n}$ we introduce a \textit{refinement interpolation operator} $\IntRef \colon S_h^{n-1} \rightarrow S_h^n$, in a similar fashion as in \cite{KL25}. 
More details on the construction for evolving surfaces is given within the construction of the common triangulation in Section~\ref{sec:smallest_common_triangulation}.

Utilising the refinement interpolant we state the full discrete method: Given $U_h^0 \in S_h^0$, determine $U_h^n \in S_h^n$ for $n = 1,\dots,K$ such that 
\begin{align}
	\label{eq:fulldiscreteheat}
	\left(\frac{U_h^n-\IntRef U_h^{n-1}}{\tau^n},\Phi_h^n \right)_{L^2(\Ga_h^n)} + (\nabla_{\Gamma_h^n} U_h^n, &\nabla_{\Gamma_h^n} \Phi_h^n)_{L^2(\Ga_h^n)} \nonumber \\
	+ (U_h^n,\Phi^n_h \nabla_{\Gamma_h^n} \cdot V_h^n)_{L^2(\Ga_h^n)} &= (F_h^n, \Phi_h^n)_{L^2(\Ga_h^n)} \qquad \forall \Phi_h^n \in S_h^n .
\end{align}

The discrete sequence of solutions $(U_h^n)_{n=1}^K$ has to be extended temporally to be comparable to the exact solution $u$. 
Following \cite{DzEll12_Fully_Disc_ESFEM}, we trivially extend any finite element function $W_h^n \in S_h^n$ constantly along the vertex-flow, which we denote by an underscore:
\begin{align*}
\underline{W_h^n} (\cdot,s) := \sum_{j=1} ^N (W_h^n)_j \Phi^n_j(\cdot,s).
\end{align*}
With the nodal basis functions $\{\Phi^n_j(\cdot,s)\}$ of the semidiscrete formulation. The extension $\underline{W_h^n}$ of the discrete $W_h^n$ thus is a constant push-forward (or pull-back) along the discrete evolving mesh $\Ga_h^n(t)$ (based on the node set at time level $n$).
We denote $S_h^n(t)$ as the FEM space of $\Ga_h^n(t)$ based on the node set $G(X_j^n,t^n,t)$. Thus $\underline{W_h^n} \in S_h^n(t)$.
Note that we can also extend functions on $\Ga$ along the exact flow given by $G$.

Using this extension we introduce two linear affine interpolations, as in \cite[Section~3.2]{KL25}, one of which is continous and defined on $\Gamma$, the other discrete and easy to lift:
The continous in time lifted discrete solution: 
\begin{equation}
\label{eq:lifted_discrete_sol}
	u_{h,\tau} (x,t) := \frac{t-t^{n-1}}{\tau^n} \big( \underline{U_h^n} \big)^{\ell^n(t)} (x,t) + \frac{t^n -t}{\tau^n} \big( \underline{U_h^{n-1}} \big)^{\ell^{n-1}(t)} (x,t) , \qquad t \in [t^{n-1},t^n], \quad x \in \Ga(t),
\end{equation}
for $n = 1,\dots,K$.
And the piecewise defined discrete function, with an additional time variable for later analysis:  
\begin{equation}
\label{eq:fully discrete solution definition}
	\baruhtau(x,t,s) := \frac{t-t^{n-1}}{\tau^n} \underline{U_h^n}(x,s) + \frac{t^n -t}{\tau^n} \underline{\IntRef U_h^{n-1}} (x,s) , \quad s,t \in (t^{n-1}, t^n],\qquad x \in \Ga_h^n(s),
\end{equation}
for $n = 1,\dots,K$ and $\baruhtau(x,0,0) = U_h ^0(x)$. The discrete object can easily be lifted as $\baruhtau(x,t,s)^{\ell^n(s)} =: \baruhtaulift(x,t,s)$ following the usual convention.
For $\baruhtau$ we introduced a second time variable as it will simplify the main ideas in the upcoming analysis. Note that the additional variable of \eqref{eq:fully discrete solution definition} can be seen as the extension of $\baruhtau$ evaluated at some time $t$ and shifted onto $\Ga_h^n(s)$. We will omit the final argument if $s=t$.

Observe that, by the transport property of the basis functions, the discrete material derivative of $\baruhtau(x,t)$, simplifies for $t \in (t^{n-1},t^n]$, pointwise on $\Ga_h^n(t)$ to
\begin{align*}
\label{eq:disc_derivative}
\partialmh \baruhtau(x,t) = \frac{\underline{U_h^n}(x,t)-\underline{\IntRef U_h^{n-1}}(x,t)}{\tau^n}.
\end{align*}
This matches the discrete weak formulation \eqref{eq:fulldiscreteheat} for $t = t^n$.

Following \cite{DzEll12_Fully_Disc_ESFEM}, for $t \in (t^{n-1},t^n]$ and a point $X(t) \in \Ga_h^n(t)$ moving along the discrete velocity $V_h$, we define the induced discrete material velocity on $\Ga(t)$ by $v_h(X^{\ell^n(t)}(t),t) := \frac{\d}{\d t} X^{\ell^n(t)}(t)$. Its corresponding material derivative reads element-wise for all $w \in H^1(\Ga(t))$
\begin{align*}
\partialmh w := \partial_t w + v_h \cdot \nabla_{\Ga(t)} w.
\end{align*}
Note that the definition of $v_h$ depends on the node set of $\Ga_h^n$. On each time interval this is fixed. Although $\baruhtaulift$ is not differentiable in time, in an interval-wise setting we have $\partial_h^\bullet \baruhtaulift = 0$ as in \cite[Section~2.2]{DzEll12_Fully_Disc_ESFEM}. We highlight that $\partial^\bullet \baruhtaulift$ does, in general, not simplify like this.
In contrast to our convention, but following the original notation \cite[Eq.~2.12]{DzEll12_Fully_Disc_ESFEM}, the induced discrete velocity is not equivalent to the lifted interpolated discrete velocity, i.e.\ $v_h \neq V_h^\ell$.

\subsection{Definition and bounds of bilinear forms}
\label{sec:def_bilinear_forms}
To compactly state the a posteriori error analysis we introduce a set of bilinear forms from \cite{Dziuk2013FiniteEM,KovacsHighOrder}. 
For functions $W, \Phi \in H^1(\Ga_h(t))$, their respective lifts to $\Ga(t)$ are $w = W^{\ell(t)}$, $\phi = \Phi^{\ell(t)}$, and the continuous velocity $v$ and interpolated discrete velocity $V_h$ we define the bilinear forms:
\begin{align*}
    m^t(w,\phi) &= \int_{\Gamma(t)} w\phi, & m_h^t(W,\Phi) &= \int_{\Gamma_h(t)} W\Phi, \\
    a^t(w,\phi) &= \int_{\Gamma(t)} \nabla_{\Gamma(t)} w \cdot \nabla_{\Gamma(t)} \phi, & a_h^t(W,\Phi) &= \int_{\Gamma_h(t)} \nabla_{\Gamma_h(t)} W \cdot \nabla_{\Gamma_h(t)} \Phi, \\
    g^t(v; w,\phi) &= \int_{\Gamma(t)} (\nabla_{\Gamma(t)} \cdot v(t)) w\phi, & g_h^t(V_h;W,\Phi) &= \int_{\Gamma_h(t)} (\nabla_{\Gamma_h(t)} \cdot V_h(t)) W\Phi, \\
    b^t(v;w,\phi) &= \int_{\Gamma(t)} \mathcal{B}(v(t)) \nabla_{\Gamma(t)} w \cdot \nabla_{\Gamma(t)} \phi, & b_h^t(V_h;W,\Phi) &= \int_{\Gamma_h(t)} \mathcal{B}_h(V_h(t)) \nabla_{\Gamma_h(t)} W \cdot \nabla_{\Gamma_h(t)} \Phi,
\end{align*}
the components $i,j = 1,2,3$ of the tensors are given by:
\begin{align*}
\mathcal{B}(v(t))|_{ij} &:= \delta_{ij} (\nabla_{\Gamma(t)} \cdot v(t))-\big((\nabla_{\Gamma(t)})_i v_j(t) +(\nabla_{\Gamma(t)})_j v_i(t) \big), \\
\mathcal{B}_h (V_h(t))|_{ij} &:= \delta_{ij} (\nabla_{\Gamma_h(t)} \cdot V_h(t))-\big((\nabla_{\Gamma_h(t)})_i (V_h(t))_j +(\nabla_{\Gamma_h(t)})_j (V_h(t))_i \big).
\end{align*}
In the above definitions the velocities are separated by a semicolon, in order to clearly indicate the main variables of the bilinear forms. Further the superscript labels the evaluation point of all functions and the integral domain, unless it is specified otherwise.

Based on the Leibniz formula \eqref{eq:leibniz_formula}, the derivatives of bilinear forms are given as.
\begin{equation}
\label{eq:form_derivatives}
\begin{aligned}
    \frac{\d}{\d t} m^t(w,\phi)
    &= m^t(\partial^\bullet w, \phi) + m^t(w, \partial^\bullet \phi) + g^t(v;w,\phi), \\
    \frac{\d}{\d t} a^t(w,\phi)
    &= a^t(\partial^\bullet w, \phi) + a^t(w, \partial^\bullet \phi) + b^t(v;w,\phi).
\end{aligned}
\end{equation}
Similar formulas hold for discrete forms.
Additionally the material derivative, can in accordance to swapping $v$ with $v_h$, be exchanged by the discrete material derivative $\partialmh$. 

The following geometric bounds in \ref{prop:geo_err_bilin_form} between continous and discrete bilinear forms are based on the results of \cite[Lemma~5.5]{Dziuk2013FiniteEM}, for the first two inequalities, and of \cite[Lemma~7.5]{Lubich_Mansour2015} for the last two. We modified the statement via norm equivalence arguments \eqref{eq:norm_equiv} to fit our framework. 
Inserting the lifted discrete objects in the bilinear forms, we bound the difference of the continuous and discrete forms in the following sense:
\begin{proposition}
	\label{prop:geo_err_bilin_form}
	Let $W, \Phi \in H^1(\Ga_h(t))$, their respective lifts to $\Ga(t)$ are $w = W^{\ell(t)}$, $\phi = \Phi^{\ell(t)}$, and the discrete velocities are $v_h$ on $\Ga(t)$ and $V_h$ on $\Ga_h(t)$ described in Section~\ref{sec:discretization}, then the following geometric bounds hold for a generic constant $c>0$ which is independent of $h$ and $t$ but depends on $\Ga(t)$.
	\begin{align*}
		|m^t(w,\phi)-m_h^t(W,\Phi)| &\leq ch^2 \|W\|_{L^2(\Gamma_h(t))} \|\Phi\|_{L^2(\Gamma_h(t))},  \\
		|a^t(w,\phi)-a_h^t(W,\Phi)| &\leq ch^2 |W|_{H^1(\Gamma_h(t))} |\Phi|_{H^1(\Gamma_h(t))},  \\
		|g^t(v_h;w,\phi)-g_h^t(V_h;W,\Phi)| &\leq ch^2 \|W\|_{L^2(\Gamma_h(t))} \|\Phi\|_{L^2(\Gamma_h(t))},  \\
		|b^t(v_h;w,\phi)-b_h^t(V_h;W,\Phi)| &\leq ch^2 |W|_{H^1(\Gamma_h(t))} |\Phi|_{H^1(\Gamma_h(t))}.
	\end{align*}
\end{proposition}

From now on, if the bilinear forms are evaluated at a discrete timestep we abbreviate $m^{t^n} =: m^n$.

\subsection{Smallest common refinement}
\label{sec:smallest_common_triangulation}

Finally to state the main result of this paper we introduce the notion of smallest common refinements based on the ideas of \cite{Kreuzer12} in the Euclidean case and extend the results for stationary surfaces of \cite[Section~3.3]{KL25}. 

Given a discrete mesh, which, by our assumptions of using the exact flow, will always be an interpolation of the surface $\Ga(t)$ for all times $t \in [0,T]$, we extend the mesh to $\Gamma_h^n(t)$  by keeping the connectivity and moving the nodes $X \in \Gamma_h^n$ along the flow \eqref{eq:flow_ode} as described in Section~\ref{ch:subsec:full_dis_time_int}.

Utilising the flow map, to push discrete node sets to a single reference domain, allows us to formulate the basis for a common refinement for some fixed time $t \in (t^{n-1},t^n]$ as in the non-evolving case. We define the smallest common refinement of subsequent meshes $\Gamma_h^{n-1}(t)$ and $\Gamma_h^n(t)$ as $\Gamma_h^{\comesh}(t) := \Gamma_h^{n-1}(t) \oplus \Gamma_h^n(t)$.

Additionally, we define the common finite element space $S_h^{\comesh}(t)$. Finally, we introduce the interpolation operators $\mathcal{I}_k^{\comesh} \colon S_h^k(t) \rightarrow S_h^{\comesh}(t)$ for both $k = n-1$ and $k=n$, respectively.
This interpolation is constructed based on the nodal values of the parent meshes: 
\begin{enumerate}
	\item[--] For nodes belonging to both $\Gamma_h^{n-1}(t)$ and $\Gamma_h^{\comesh}(t)$, and for nodes belonging to both $\Gamma_h^n(t)$ and $\Gamma_h^{\comesh}(t)$, respectively, the nodal values are kept.
	\item[--] For nodes which are missing in either $\Gamma_h^{n-1}(t)$, or which are missing in $\Gamma_h^n(t)$, respectively, the nodal value is assigned by the Lagrangian interpolation.
\end{enumerate}

In comparison to the stationary surface case \cite{KL25} the movement introduces further difficulties, which adds further complexity to the vertex and element matching.
This can be seen by the mismatch of the lifted nodes and the ones we flow in time to generate the common triangulation, see Figure~\ref{fig:common_triangulation} where the nodes marked by cross ("$x$") do not align with the flown nodes coming from $\Gamma_h^{n-1}$.
Assuming that the refinement and coarsening process is handled carefully, it is possible to reduce the problem, such that the vertices align under the flow, to the stationary case. 
Then the $\oplus$ operator matches precisely with the one introduced in \cite[Section~3.3]{KL25} with the main idea to keep matching elements/vertices and always taking the most refined elements if one of the meshes is locally more refined.

Note that in general the common refinement is not necessarily a refinement of the underlying meshes, due to the nonlinear lifting process required for keeping the interpolation property of the mesh for all times. For the evolving case, where lifts are taken at two discrete time-levels, this is even more complex as the vertex correspondences are in general not related by a set of consecutive lifts.

\textbf{Guaranteeing node alignment for consecutive meshes}
The main issue with the construction is that vertices are constructed linearly but are getting misaligned by the non-commuting and non-linear discrete lifts $\ell^{n-1}$ and $\ell^n$ and flow map $G$. 
In particular if ones compares the bottom right triangulation in Figure~\ref{fig:common_triangulation}, with the non-matching vertices (where the grey circle and the black cross ("$x$") are slightly apart). The black cross results from push-forward to time $t^n$ and the respective lift $\ell^{n}$, whereas the grey vertex results from lift at time $t^{n-1}$ via $\ell^{n-1}$ and then push-forward to $t^n$.
The construction, which guarantees node alignment for consecutive meshes is depicted in Figure~\ref{fig:common_triangulation}.
\begin{figure}[htbp]
    \centering
    \includegraphics[width=1\linewidth]{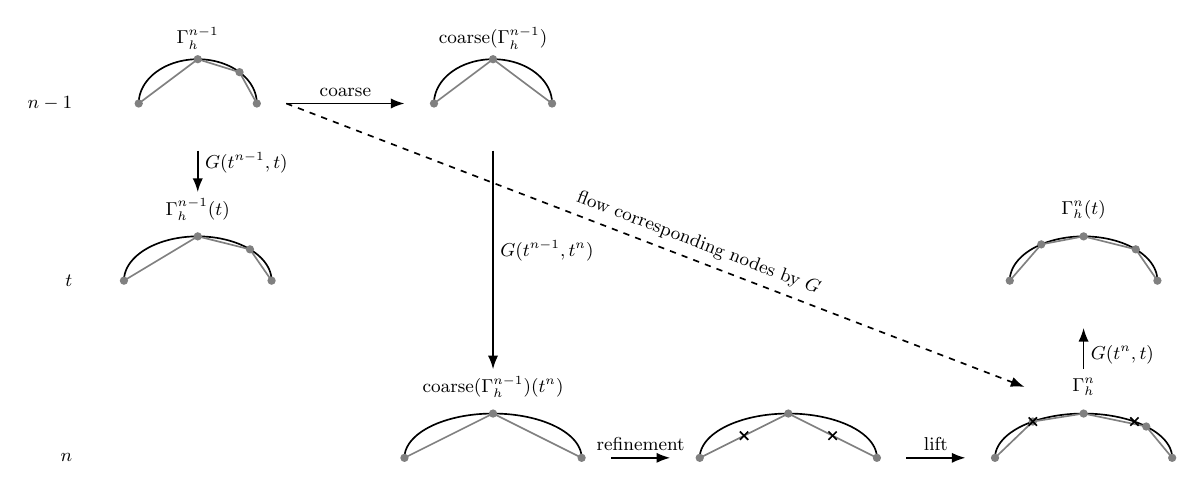} 
    \caption{Construction of triangulations to guarantee node alignment for consecutive meshes. Each row corresponds to the time $t^{n-1}$, $t$ and $t^n$ from top to bottom. The meshes in the middle row, i.e.\ $\Ga_h^{n-1}(t)$ and $\Ga_h^n(t)$ are the basis for the common triangulation.}
	\label{fig:common_triangulation}
\end{figure}

We can control the vertex movement in a semi-local (explained in Remark~\ref{rem:semi_local_common_mesh}) fashion by assuming the following mesh generation steps:
\begin{itemize}
\item There is an initial coarse admissible triangulation $\Ga_h^{\text{init}}$ whose nodes are never coarsened. This means that all vertices $X_i$ of $\Ga_h^{\text{init}}$ are also present as vertices $G(X_i,0,t^n)$ in $\Gamma_h^n$ for all discrete timesteps.
\item Provided we determined a solution $u_h^{n-1}$ based on \eqref{eq:fulldiscreteheat} with its corresponding mesh $\Ga_h^{n-1}$ (top left of Figure~\ref{fig:common_triangulation}), we apply a coarsening step based on newest-vertex-bisection (NVB), resulting in $\text{coarse}(\Ga_h^{n-1})$ (top right of Figure~\ref{fig:common_triangulation}). The mesh is transported by the flow $G$ to time $t^n$ to initialize the first guess of a mesh $\Ga_h^n$ (bottom Left of Figure~\ref{fig:common_triangulation}) to solve the subsequent step, i.e.\ vertices are moved along the flow and connectivity is unchanged.
\item The refinement process updates $\Ga_h^n$ until the provided solution is accepted in the adaptive routine. All refinements are based on NVB and to handle vertex mismatch the process is divided in two cases:
\begin{enumerate}
\item A refined node is constructed by applying the refinement, creating intermediate nodes (see bottom middle of Figure~\ref{fig:common_triangulation}) which in general do not interpolate the surface $\Ga(t^n)$ and then lifted based on $\ell^n$ to regain the admissibility of the mesh (nodes marked by crosses in bottom right of Figure~\ref{fig:common_triangulation}). Note that the intermediate nodes are the evaluation points used to define the refinement interpolant $\IntRef$.
\item But if a refined node corresponds to a vertex which was just coarsened in the initialization step the construction is different. Instead of constructing a new vertex we update $\text{coarse}(\Ga_h^{n-1})$ by reappending the corresponding vertex and connectivity. Efficiently we refine by storing the new connectivity and then update the coordinates for just coarsened nodes by flowing the corresponding vertices at time $t^{n-1}$ with the flow $G$ to time $t^n$ (see long dashed arrow from top left to bottom right of Figure~\ref{fig:common_triangulation}).  
\end{enumerate}
\end{itemize}
These assumptions suffice to construct the smallest common refinement as it guarantees that corresponding nodes always align under the flow.
Note however that we only avoid the mismatch under the flow for two subsequent meshes but for any set of discrete meshes this property fails in general.
Further note that the NVB guarantees that the refinement hierarchy is unique, which implies that the connectivity and how elements are bisected is always identical.

\begin{remark}
The initial coarse triangulation is enforced as a technical tool to guarantee that the closest point projection is bijective for all times.
Further we always work with NVB, which is the typical choice for adaptive finite elements on surfaces (see, e.g.\, \cite{Demlow2007, KL25}),
but general successive bisections, red-refinement and other strategies fulfilling Conditions 3,4 and 6 of \cite{Bonito_Nochetto_2010_mesh_cond} are valid candidates for refinement strategies on surfaces as discussed in \cite{Bonito2013}.
\end{remark}

\begin{remark}
\label{rem:semi_local_common_mesh}
We have seen that the construction of the common mesh is affected by the movement of the mesh, in particular that two subsequent meshes introduce two distinct lifts, which is used in every subsequent refinement making identification of nodes non-trivial. We discuss two further options and argue why the current construction is chosen.

1. Instead of dealing with different lifts it would also be possible to base all refinement on a singular base mesh (say at time $t= 0$) and then flow the points along the exact flow. However this is in general not desirable as geometric features could be better resolved if we work with the mesh at the current timestep instead. Further, constantly flowing points from the initial surface can be costly, in particular if the flow of vertices is approximated by a time-stepping scheme.

2. It is theoretically possible to allow the mismatch and base all refinements on their respective current discrete timestep. This requires the analysis of the resulting perturbations.
If we focus on a singular vertex which was constructed as a lift at time $t^{n-1}$, then coarsened and later re-refined at time $t^n$ with corresponding lift, and flow both meshes to time $t$, we observe a mismatch. However for sufficiently small timesteps and the assumption that $G \in C^2$, the coordinate mismatch can be shown to be $\mathcal{O}(h^2)$ where $h$ is the local element size. If one carefully extends the analysis of \cite{KL25} the new resulting error would be an error between two identical finite element functions but where a subset of nodes are perturbed by $\mathcal{O}(h^2)$. The resulting error is a geometric error and one can show that it has the same orders as the other geometric indicators but as the functions are still compared on $\Ga(t)$ an arbitrary point $a \in \Ga(t)$ unlifts to possibly different elements in the two distinct perturbed triangulations. This situation is similar to Figure~\ref{fig:mismatch_common_triangulation}, however without guarantee that the unlift is inside a parent element. This requires us to introduce patch-wise estimates which makes the coarsening indicator complex to evaluate. 

Our restrictions allow us to avoid node mismatches but without enforcing that refinement is done on a singular mesh. Further the refinement and coarsening process is purely local and only requires information at the current timestep, but as required for the analysis the initial macro-triangulation is always kept which could lead to suboptimal node positioning. 
We refer to Section~\ref{sec:dumbell} for further discussion.
\end{remark}

\section{Main results}
\label{ch:main_result}
To derive reliable and efficient error indicators (up to oscillation, coarsening, mesh-transfer, high-order geometric contributions and flow consistency errors) we employ residual-based error analysis. The resulting indicators allow for space--time adaptivity (see Section~\ref{sec:numerical_experiments}) for our model problem of a parabolic PDE on an evolving surface \eqref{eq:heat_strong}. Note that the following results are expected to be extendible to more general parabolic PDEs on evolving surfaces.

We start by formulating the main results, including the equivalence of error and residual, and the upper and lower bounds to the error by given error indicators.
As is typical for residual-based error analysis \cite{Verfuerth2003, Kreuzer12,KL25}, we excluded the data oscillation term in the a posteriori error analysis. 
We highlight that the efficiency results are restricted to the spatial and temporal residuals, which entail the key local information to resolve the PDE.
It is usual to exclude terms arising from the coarsening and geometric residuals in the efficiency analysis, as they do not dominate the errors, see the discussion in \cite[Section~4]{KL25}.

However, there are additional consistency terms for which we only show reliability, namely the velocity-induced temporal indicator and the movement indicator, both resulting due to the flow of the surface. We will see that these terms ensure that the functions are transported sufficiently well in time, in particular preserving the correct mass transport in areas of high velocity divergence. The resulting indicators will yield an additional tool to control the time-step size, in fact it would even be possible to have local time-stepping control based on these indicators, however we will not follow this approach and assume adaptive global time-stepping. In our proposed algorithm, see Section~\ref{sec:numerical_experiments}, the velocity-induced indicator and the movement indicator will be used to determine a first guess for a  sufficiently small step-size to ensure good mass transport, which is then possibly further refined in the adaptive interplay of spatial and temporal refinements by the residual components directly related to the PDE.
Also algorithmically we assess whether the mesh-transfer between timesteps dominates the error, in the numerical examples of Section~\ref{sec:numerical_experiments}, this was not the case.

To formulate the main results we state the following bound and definitions:
By $c$ we will always denote a generic positive constant, that is independent of $h$, $\tau^n$, and $n$, but may change its value between steps. 
Note that the standard regularity assumption of the flow map \eqref{eq:flow_ode} \cite{dziuk_elliott_2012, Dziuk2013FiniteEM}, implies for the following bounds on the velocity 
\begin{align}
\label{eq:kappa}
\|\nabla_{\Gamma(t)} \cdot v(t)\|_{L^\infty(\Gamma(t))} + \|\mathcal{B}(v(t))\|_{L^\infty(\Gamma(t))} \leq \kappa \qquad \forall t \in (0,T],
\end{align}
with a time independent $\kappa \in \R$.
Restricted to $t \in (t^{n-1},t^n]$, we define the possibly smaller interval-wise $\kappa^n$, in the same manner. 
Additionally we define the $\kappa$-dependent a global constant by $C_\kappa = \exp(\kappa T)$ and an interval-wise constant by $C_\kappa^n := \exp(\kappa^n \tau^n)$. 
Note that these bounds extend, possibly with an additional scalar constant, to the discrete interpolated velocity $V_h$ (see \cite[Remark~3.3]{dziuk_elliott_2012}) and discrete velocity $v_h$ by interpolation estimates (see \cite[Lemma~5.6]{DziukElliott2013_L2estimates}) for sufficiently small $h \leq h_0$.

For $t \in (0, T]$ and $w \in H^1(\Gamma(t))$ the residual is defined by:
\begin{equation}
	\label{eq:res}
	\langle \mathcal{R}(u_{h,\tau}),w\rangle = m^t(\partial^\bullet u_{h,\tau},w)+a^t(u_{h,\tau},w)+g^t(v;u_{h,\tau},w)-m^t(f,w).
\end{equation}
The key proposition is the equivalence of error and residual. To formulate the proposition we define the natural norm for parabolic PDEs on evolving surfaces with $s,t \in [0,T]$:
\begin{equation}
	\label{eq:graph_norm}
	\|w\|_{X(s,t)}^2 := \|w\|_{L^\infty (s,t; L^2 (\Ga(\cdot)))}^2 +\|w\|_{L^2 (s,t; H^1(\Ga(\cdot)))}^2 + \|\partialm w\|_{L^2 (s,t; H^{-1} (\Ga(\cdot)))}^2 .
\end{equation}
The notation is understood in the sense of \cite{AlphonseElliottStinner}, with the choice $V(t)=H^1(\Gamma(t))$, $H(t)=L^2(\Gamma(t))$, and
$V^*(t)=H^{-1}(\Gamma(t))$ in their notation.

Thus we can state the equivalence result.
\begin{proposition}
\label{prop:equiv_error_res}
Assume the velocity satisfies \eqref{eq:kappa}, then the residual $\mathcal{R}(u_{h,\tau})$ from \eqref{eq:res} and the error, between the exact solution of \eqref{eq:weak_form} and the discrete solution \eqref{eq:lifted_discrete_sol}, in the graph norm \eqref{eq:graph_norm} obey for all $w \in L^2(0,T; H^1(\Ga(\cdot)))$
	\begin{subequations}
		\begin{align}
			\label{eq:res_equiv_2}
			\|u -  u_{h,\tau} \|_{X(0,T)} 
			\leq &\  c C_\kappa \left( \|u^0 - (u_h^0) \|^2 _{L^2(\Ga^0)} + \|\mathcal{R}(u_{h,\tau}) \|^2 _{L^2(0,T; H^{-1} (\Ga(\cdot)))} \right)^{1/2} , \\
			\label{eq:res_leq_error}
			\int_0^{T} \langle \mathcal{R}(u_{h,\tau}), w \rangle dt \leq &\  (1+\kappa) \|u- u_{h,\tau}\|_{X(0,T)} \|w\|_{L^2(0,T; H^1(\Ga(\cdot)))}.
		\end{align}
	\end{subequations}
	The constants $c > 0$ is independent of $h$, $\tau^n$, and $n$, but depends exponentially on the final time $T$.
\end{proposition}

\subsection{Indicators}
We define a set of error indicators and the oscillation for time interval $t \in (t^{n-1},t^n]$:
\begin{subequations}
\label{eq:indicator}
	\begin{align}
		\label{eq:indicator - full}
		\eta^n &\phantom{:}= \Big( \tau^n \Big( (\eta_{h}^n)^2 +  (\eta_{\tau}^n)^2+ (\zeta_{\tau}^n)^2 + (\zeta_{\textnormal{move}}^n)^2 + (\eta_\textnormal{c}^n)^2 + (\eta_\textnormal{trans}^n)^2 + (\mathcal{G}_h^n)^2 +(\mathcal{G}_v^n)^2  \Big) \Big)^{\frac{1}{2}}, \\
		&\text{where the individual indicators are defined by} \nonumber \\
		\label{eq:indicator - spatial}
		(\eta_{h}^n)^2 :&= \! \sum_{S \in \mathcal{S}_h^n} \!\! h_S \big\|\llbracket \nabla_{T} U_h^n \cdot \mathrm{n}_S \rrbracket \big\|_{L^2(S)}^2 \! \nonumber \\
		& \qquad+  \!\! \sum_{T \in  \mathcal{T}_h^n} \!\! h_T^2 \Big\|\frac{1}{\tau^n} (U_h^n - \IntRef U_h^{n-1})+ (\nabla_T \cdot V_h^n) U_h^n  -F_h^n \Big\|_{L^2(T)}^2, \\
		\label{eq:indicator - temporal}
		(\eta_{\tau}^n)^2 :&= \! \sum_{T \in \mathcal{T}_h^n} |U_h^n - \IntRef U_h^{n-1}|_{H^1(T)}^2  , \\
		\label{eq:indicator - temporal2}
		(\zeta_{\tau}^n)^2 :&= \! \sum_{T \in \mathcal{T}_h^n} \|(\nabla_T \cdot V_h^n)(U_h^n - \IntRef U_h^{n-1})\|_{L^2(T)}^2  , \\
		\label{eq:indicator - move}
		(\zeta_{\textnormal{move}}^n)^2 :&= \sum_{T \in \mathcal{T}_h^n} (\kappa^n)^2 (\tau^n)^2 \Bigg(\Big\|\frac{1}{\tau^n} (U_h^n-\IntRef U_h^{n-1}) + (\nabla_{T} \cdot V_h^n) U_h^n -F_h^n\Big\|^2_{L^2(T)} + |U_h^n|^2_{H^1(T)} \Bigg), \\
		\label{eq:indicator - coarse}
		(\eta_{\textnormal{c}}^n)^2 :&= \! \sum_{T \in \mathcal{T}_{h,\textnormal{coarse}}^{\comesh}} \Big((\kappa^n)^2+\frac{1}{(\tau^n)^2}\Big) \big\|\underline{\mathcal{I}_{n}^{\comesh} \underline{\IntRef U_h^{n-1}} -\mathcal{I}_{n-1}^{\comesh}\underline{U_h^{n-1}}} \big\|_{L^2(T)}^2 \nonumber \\ 
	&\quad \qquad \qquad + \big|\underline{\mathcal{I}_{n}^{\comesh} \underline{\IntRef U_h^{n-1}} -\mathcal{I}_{n-1}^{\comesh}\underline{U_h^{n-1}}}\big|_{H^1(T)}^2 \\ 
		\label{eq:indicator - refinement}
	(\eta_{\textnormal{trans}}^n)^2 :&= \! \sum_{T \in \mathcal{T}_{h,\textnormal{ref}}^{n-1}} \bigg( \big(\frac{h_T^2}{\tau^n}\big)^2+h_T^4 (\kappa^n)^2 + h_T^2\bigg) |U_h^{n-1}|_{H^1(T)}^2 \nonumber \\
	&\quad + \sum_{T \in \mathcal{T}_{h,\textnormal{coarse}}^{n}} \bigg( \big(\frac{h_T^2}{\tau^n}\big)^2+h_T^4 (\kappa^n)^2 + h_T^2\bigg)  |\IntRef U_h^{n-1}|_{H^1(T)}^2, \\
		\label{eq:indicator - geo_h}
		(\mathcal{G}_h^n)^2 :&= \!  \sum_{T \in \mathcal{T}_h^n} h_T^4 |U_h^n|_{H^1(T)}^2 , \\
		\label{eq:indicator - geo_v}
		(\mathcal{G}_v^n)^2 :&= \! \sum_{T \in \mathcal{T}_h^n} h_T^4 \big( \|U_h^n\|_{L^2(T)}^2 + \|\IntRef U_h^{n-1}\|_{L^2(T)}^2 \big) ,\\
		\label{eq:oscillation}
		\osc^n :&= \!  f-(\underline{F_h^n})^{\ell^n(t)} + \big((\nabla_{\Ga} \cdot v_h)-(\underline{\nabla_{\Ga_h^n} \cdot V_h^n})^{\ell^n(t)} \big) \baruhtaulift.
	\end{align}	
\end{subequations}
Where the set of all edges at time $t^n$ is denoted by $\mathcal{S}_h^n$, the jump across an edge $S \in \mathcal{S}_h^n$ is given by $\llbracket w \rrbracket|_S := w|_{T_1} - w|_{T_2}$, where $T_1$ and $T_2$ are the two triangles from $\mathcal{T}_h^n$ sharing the edge $S$. Further on a discrete element $T$ the tangential gradient is denoted by $\nabla_T$, the outward edge-normal $\mathrm{n}_S$ of $S$ is defined with respect to $T_1$, and the diameter of an element is labelled $h_T$. 

We define the set of coarsened elements $\mathcal{T}_{h,\text{coarse}}^{n}$ as the set of elements of $\Ga_h^{n}$, which do not coincide with any element in the common triangulation $\Ga_h^{\comesh}(t^n)$. Conversely, we can define $\mathcal{T}_{h,\text{coarse}}^{\comesh}$ by swapping roles of $\Ga_h^n$ and $\Ga_h^{\comesh}(t^n)$.
 Analogously, the set of refined elements $\mathcal{T}_{h,\text{ref}}^{n}$ is the set of elements of $\Ga_h^{n-1}(t^n)$, which do not coincide with any element in the common triangulation $\Ga_h^{\comesh}(t^n)$.

These indicators define the global in space and local in time indicator $\eta^n$. We refer to the usual $\eta_h^n$, $\eta_\tau^n$, and  $\eta_{\textnormal{c}}^n$ as spatial, temporal and coarsening indicator respectively.
The indicators $\zeta_\tau^n$, $\zeta_\textnormal{move}^n$ and $\mathcal{G}_v^n$ are the velocity induced temporal, movement and velocity induced higher-order geometric indicator required due to movement of the surface. We refer to the first two as consistency errors of the flow and the last one is often just referred to as the geometric error, as it behaves just as the other high-order contribution $\mathcal{G}_h^n$, which arises in standard a posteriori error analysis for elliptic and parabolic surface PDEs, see, e.g.\, \cite{Demlow2007,Bonito2013,Camacho2014} and \cite{KL25} respectively.
The mesh-transfer indicator $\eta_{\textnormal{trans}}^n$, was present in \cite{KL25}, within the coarsening indicator but now separated. This separation is done to handle the mesh-transfer contribution explicitly algorithmically (see Section~\ref{sec:numerical_experiments} for further details).
Finally, the oscillation $\osc^n$ includes the typical right-hand side control and additionally a corresponding velocity oscillation control.

\begin{remark}
In comparison to the analysis on stationary surfaces in \cite{KL25} there are additional non-negligible contributions, in particular $\zeta_\tau^n$ and $\zeta^n_\textnormal{move}$.
But if we assume that $v = 0$ and thus both $v_h = 0$ and $\kappa^n = 0$ for all $n$, all indicators almost directly collapse to their stationary variant, see, f.ex.\, the spatial indicator \eqref{eq:indicator - spatial}. 
The only contribution which does not readily collapses is the higher order $\mathcal{G}^n_v$ which however, trivially vanishes by analysing the corresponding residual \eqref{eq:res_decom}. 
\end{remark}

\begin{remark}
Note that $\kappa^n$ is not directly computable, or at least not without considerable computational effort. Therefore, instead of working with $\kappa^n$ itself in the numerical experiments of Section~\ref{sec:numerical_experiments}, we assume that the velocity oscillation is resolved well enough such that $\|\nabla_{\Ga_h^n} \cdot V_h^n\|_{L^\infty} + \|\mathcal{B}_h^n\|_{L^\infty}$ provides a reliable approximation of $\kappa^n$ in \eqref{eq:kappa} on each time subinterval. This corresponds to a standard oscillation-type argument: the involved velocity-dependent quantities are assumed to vary only mildly within each timestep and are therefore well approximated by their values at the discrete time levels.
\end{remark}

\subsection{Main result: reliability and efficiency} 
We relate the error indicators \eqref{eq:indicator} to the errors, showing the reliability for the full indicator and efficiency for the consistency parts of the PDE. 
\begin{theorem}
\label{thm:upper_lower}
	Let $h \leq h_0$ and $\tau \leq \tau_0$, with sufficiently small $h_0 > 0$ and $\tau_0 >0$, the residual-based error estimator $\eta^n$ of \eqref{eq:indicator}, and the error between the solution $u$ of \eqref{eq:weak_form} and the numerical approximation $u_{h,\tau}$ \eqref{eq:lifted_discrete_sol}, obtained via \eqref{eq:fulldiscreteheat}, satisfies the following estimates for $0 < t^n = \tau^1 + \dotsb + \tau^n \leq T$:
	
	(a) A global upper bound in space and time (reliability up to oscillation):
	\begin{subequations}
		\begin{align}
			\label{eq:reliability estimate}
			\|u-u_{h,\tau} \|_{X(0,t^n)} \leq c^\star C_\kappa \left(\sum_{j=1}^n (\eta^j)^2 + \|\osc^j \|_{L^2(t^{j-1},t^j; H^{-1} (\Ga(\cdot)))}^2 + \|u^0-(u_h^0)^\ell \|_{L^2(\Ga^0)}^2 \right)^{\frac{1}{2}} . \\
		\intertext{\indent (b) A lower bound which is global in space and local in time (efficiency up to oscillation, geometric, coarsening and mesh-transfer defects, and the velocity scaled temporal indicator):}
		\label{eq:efficiency estimate}
			\eta^n \leq  c_\star C_\kappa^n \Big( \|u-u_{h,\tau} \|_{X(t^{n-1},t^n)} +\|\osc^n \|_{L^2(t^{n-1},t^n; H^{-1} (\Ga(\cdot)))} +  (\tau^n)^{\frac{1}{2}} \big(\mathcal{G}_h^n+\mathcal{G}_v^n + \eta_\textnormal{c}^n + \eta_\textnormal{trans}^n  + \zeta_\tau^n + \zeta_\textnormal{move}^n \big) \Big). 
		\end{align}
	\end{subequations}
	The constants $c_\star > 0$ and $c^\star > 0$ are independent of $h$, $t^n$, and $\tau^n$, but depend on the shape-regularity constant $\varrho^n$ of $\Ga_h^n$, and on $\Ga$. The constant $c^\star$ additionally depends on the shape-regularity constants $\varrho^j$ of the prior meshes $\Ga_h^j$. The constants $C_\kappa, C_\kappa^n$ are given in \eqref{eq:kappa}.
\end{theorem}

Theorem~\ref{thm:upper_lower} will be proved in the subsequent section.

\begin{remark}
\label{remark:h0 sufficiently small}
	In an adaptive setting the assumption $h \leq h_0$ might seem counter-intuitive, as it restricts coarsening. It is, however, inevitable to ensure that the closest point projection \eqref{eq:uniquedecom} is unique. It is further required by all geometric approximations, see, e.g.\, ~\cite{Dziuk2013FiniteEM}. The constant $h_0$ solely depends on the curvature of $\Ga(t)$, and it enforces that throughout the adaptivity one cannot coarsen beyond some suitable triangulation, where the lift is bijective.
\end{remark} 

\begin{remark}
\label{remark:Geometric term and coarsening indicator}
The above theorem does not include efficiency for the high-order geometric terms, coarsening, mesh-transfer and velocity-induced indicators. As stated in \cite[Section~4.2]{Camacho2014} the geometric term arising from the stiffness term \eqref{eq:res_decom} is not the main concern when dealing with convergence and optimality of an adaptive algorithm. As \cite[Lemma~5.8 \& Chapter~6.1]{Bonito2013}  suggests they can be handled utilising an additional adaptive routine to guarantee that the geometric errors are bounded by the spatial indicator, which infers that the lower bound \eqref{eq:efficiency estimate} holds up to oscillation. The coarsening indicator and the second term of the mesh-transfer indicator, both introduced by the coarsening residual, can be made arbitrarily small by coarsening less. 
Note that the first term of the mesh-transfer indicator \eqref{eq:indicator - refinement} only contributes for elements being refined between timesteps, which is, similar to coarsening, but more involved, algorithmically controllable (see the Algorithm described in Section~\ref{sec:numerical_experiments}).
The velocity-induced terms, in particular $\zeta_\tau$ and $\zeta_{\textnormal{move}}$, yield non-neglectable contributions to the error. We can however control these quantities by matching the temporal stepsize (based on the indicators) to the velocity divergence. The contributions are directly related to the correct mass transport and we argue that sufficient control of these indicators allows us to recover the efficiency for the consistency error of the PDE.
\end{remark}

\section{Proof of the main result}
\label{ch:proof}
The proof consists of two parts: first showing the equivalence of error and residual in Section~\ref{sec:equiv_err_res}, and then bounding the residual in terms of the indicators in the subsequent sections.
The residual bounds are simplified by splitting the residual (see Section~\ref{sec:decomposition}) into different components, related to different error sources of the discretization.
In particular, the typical spatial and temporal residual, closely related to the PDE error of the discretization on the discrete domain; the geometric residual collecting errors arsing due to the polyhedral approximation of the surface, the novel flow-induced indicators arising due to the movement of the surface, and the coarsening and mesh-transfer indicators arising due to exchange between the two time-interpolations of Section~\ref{ch:subsec:full_dis_time_int} required for computability.
After stating the norm equivalence under surface evolution in Section~\ref{sec:relating_residual_ind}, which will be used to move functions from continuous to discrete timesteps and vice versa, some residual bounds, in particular for spatial and geometric residuals are easily extendible from the stationary analysis (see Sections~\ref{sec:res_spatial}\& \ref{sec:res_geo}).
Afterwards we focus on the novelties emerging by the surface evolution, namely the movement residual in Section~\ref{sec:res_move}, which requires a substantially different approach to show reliability, and the analysis of the coarsening residual in Section~\ref{sec:res_coarse}, which requires a careful setup with tools introduced in Section~\ref{sec:smallest_common_triangulation}.
The final Section~\ref{sec:res_temp} on temporal indicators combines the prior bounds to finish the proof. 
We highlight that for all upcoming results, we tagged bounds by the labels $(a)$ or $(b)$ indicating that these results are explicitly used to show \eqref{eq:reliability estimate} and \eqref{eq:efficiency estimate} respectively.

As our main result, many results hold for a mesh size $h \leq h_0$, which is always understood with a sufficiently small $h_0 > 0$, see Remark~\ref{remark:h0 sufficiently small}. Additionally we assume that $\tau \leq \tau_0$ with a sufficiently small $\tau_0$ which allows us absorb higher-order contributions which are asymptotically neglectable.

\begingroup
\footnotesize
\setlength{\tabcolsep}{4pt}
\renewcommand{\arraystretch}{1.08}
\begin{longtable}{|p{0.35\textwidth}|p{0.40\textwidth}|p{0.15\textwidth}|}
\hline
\textbf{Description} & \textbf{Notation} & \textbf{Defined in} \\
\hline
\endfirsthead

\hline
\textbf{Description} & \textbf{Notation} & \textbf{Defined in} \\
\hline
\endhead

Evolving surface & $\Gamma(t)$ & Section~\ref{sec:prelim} \\
\hline
Discrete surface at timestep $t^n$ & $\Gamma_h^n$ & Section~\ref{ch:subsec:full_dis_time_int} \\
\hline
Lift of point/function & $X = x^{\ell}$, $W_h = w_h^{\ell}$ & Section~\ref{sec:semidiscreteESFEM} \\
\hline
Space--time manifold & $Q_T := \bigcup_{t\in[0,T]}\Gamma(t)\times\{t\}$ & Section~\ref{sec:prelim} \\
\hline
Two-time flow map & $G(x,t,s):\Gamma(t)\to\Gamma(s)$ & Section~\ref{sec:prelim} \\
\hline
Surface velocity on $\Gamma(t)$ & $v(x,t)$ & \eqref{eq:flow_ode} \\
\hline
Discrete material velocity on $\Gamma_h(t)$ & $V_h(X,t)$ & Section~\ref{sec:semidiscreteESFEM} \\
\hline
Induced discrete velocity on $\Gamma(t)$ & $v_h(x,t)$ & Section~\ref{ch:subsec:full_dis_time_int} \\
\hline
Velocity bound and related constants & $\kappa$, $\kappa^n$, $C_\kappa=e^{\kappa T}$, $C_\kappa^n=e^{\kappa^n\tau^n}$ & \eqref{eq:kappa} and below \\
\hline
Fully discrete solution & $U_h^n(x)$ & \eqref{eq:fulldiscreteheat} \\
\hline
Constant-in-flow extension & $\underline{W_h^n}(x,t)$ & Section~\ref{ch:subsec:full_dis_time_int} \\[0.6ex]
\hline
Continous time interpolation & $u_{h,\tau}(x,t)$ & \eqref{eq:lifted_discrete_sol} \\
\hline 
\rule{0pt}{2.8ex} Discrete time interpolation purely on $\Ga_h^n(s)$ & $\baruhtau (x,t,s)$ & \eqref{eq:fully discrete solution definition} \\
\hline
Refinement interpolation operator & $\mathcal{I}_{\mathrm{ref}}^n:S_h^{n-1}\to S_h^n$ & Section~\ref{ch:subsec:full_dis_time_int} \\
\hline
Continuous bilinear forms at time $t$ & $m^t(w,\phi), a^t(w,\phi), g^t(v;w,\phi), b^t(v;w,\phi)$ & Section~\ref{sec:def_bilinear_forms} \\
\hline
Discrete bilinear forms at time $t$ & $m_h^t(W,\Phi), a_h^t(W,\Phi), g_h^t(V_h;W,\Phi), b^t(V_h;W,\Phi)$  & Section~\ref{sec:def_bilinear_forms} \\
\hline
Smallest common refinement & $\Gamma_h^{\comesh}(t)=\Gamma_h^{n-1}(t)\oplus\Gamma_h^n(t)$ & Section~\ref{sec:smallest_common_triangulation} \\
\hline
Interpolation to common mesh & $\mathcal I_k^{\comesh}:S_h^k(t)\to S_h^{\comesh}(t)$ & Section~\ref{sec:smallest_common_triangulation} \\
\hline
\end{longtable}
\endgroup
\subsection{Equivalence of error and residual}
\label{sec:equiv_err_res}
We start with the proof of Proposition~\ref{prop:equiv_error_res}. 
As discussed in \cite{KL25} we have to analyse the residual with respect to the continuous in time $u_{h,\tau}$ to be able to employ standard arguments when showing the equivalence of residual and error.

\begin{proof}

	Denote the error by $e(t) := u(t)- u_{h,\tau}(t)$. For notational convenience, we drop the explicit surface dependence on $\Gamma(t)$ and $\Gamma(0)$ in the norms.
	
	(a) 
	 We utilise energy estimates to show the upper bound. Test the residual equation \eqref{eq:res} of $\mathcal{R}(e)$ with $w = e$ further note that $\mathcal{R}(u) = 0$.
	We rearrange the residual-based on \eqref{eq:form_derivatives} using $m^t(\partialm e,e) = \frac{1}{2} \frac{\d}{\d t} m^t(e,e) - \frac{1}{2} g^t(v;e,e)$, standard estimates, an absorption of the $H^1$-seminorm, and bounding the velocity-divergence by $\kappa$ (see \eqref{eq:kappa}) then gives
	\begin{align*}
	\frac{\d}{\d t} \|e\|^2_{L^2} + |e|^2_{H^1} \leq \|\mathcal{R}(u_{h,\tau}) \|_{H^{-1}}^2 +  (1+\kappa)\|e\|_{L^2}^2.
	\end{align*}
	Now we integrate over the time interval $(0,t)$
	\begin{align}
	\label{eq:equivalence_proof_eq1}
	\|e(t)\|^2_{L^2} + \int_0^t |e|^2_{H^1} \leq \|\mathcal{R}(u_{h,\tau}) \|_{L^2(0,t;H^{-1})}^2 + \|e(0)\|_{L^2}^2 + (1+\kappa) \int_0^t \|e\|_{L^2}^2.
	\end{align}
	Applying Gronwall inequality, yields the $L^\infty(L^2)$-bound
	\begin{align*}
	\|e\|^2_{L^\infty(0,t;L^2)} \leq \exp(t(1+\kappa))(\|\mathcal{R}(u_{h,\tau}) \|_{L^2(0,t;H^{-1})}^2 + \|e(0)\|^2_{L^2}).
	\end{align*}
	We can immediately follow $\|e\|^2_{L^2(0,t;L^2)} \leq T \|e\|^2_{L^\infty(0,t;L^2)}$ and thus from \eqref{eq:equivalence_proof_eq1} that $|e|^2_{L^2(0,t;H^1)}$ is also bound by the residual and initial error. Together we obtain the bound for the full $H^1$-norm. 
	
	Finally, we bound the material derivative of the error, from \eqref{eq:res} we can bound 
	\begin{align*}
	\|\partialm e \|_{H^{-1}} \leq  \|\mathcal{R}(u_{h,\tau})\|_{H^{-1}} + |e|_{H^1} + \kappa \|e\|_{L^2}. 
	\end{align*}
	Now time integration and the prior results for the $L^2(0,t;H^1)$-norm yields the upper bound.
	
	(b) Again starting from $\langle \mathcal{R} (e),w \rangle$ we employ duality and Cauchy--Schwarz inequalities and integrate over time to obtain
	\begin{align*}
	\int_0^T  \langle \mathcal{R} (u_{h,\tau}(s)),w(s) \rangle \d s = \int_0^T  \langle \mathcal{R} (e(s)),w(s) \rangle \d s \leq \big(\|\partialm e\|_{L^2(0,T;H^{-1})} + (1+\kappa) \|e\|_{L^2(0,T;H^{1})}\big) \|w\|_{L^2(0,T;H^1)}.
	\end{align*}
	The right-hand side of the inequality is controlled by the graph norm \eqref{eq:graph_norm} thus finishing the proof. \hfill
\end{proof}

\subsection{Decompositions}
\label{sec:decomposition}
Now that we established equivalence statements for error and residual we are able to focus on bounding the residual.
We continue by following the typical parabolic residual-based analysis, see, e.g.\, \cite{Verfuerth2003}, by splitting the residual in spatial and temporal contributions. Many decomposition steps are on par with \cite[Section~5.2]{KL25} where we step by step exchange continous to discrete quantities.

The continuity of $u_{h,\tau}$ was required for Proposition~\ref{prop:equiv_error_res} but it is difficult to employ standard arguments when we deal with a function which contains two different lifts. On the other hand $\baruhtaulift(x,t)$ can be written with a singular lift of the fully discrete solution \eqref{eq:fully discrete solution definition}, and exchanging the two interpolations naturally introduces a coarsening type residual: 
\begin{align}
\label{eq:residual_coarsening_decom}
	\langle \mathcal{R}(u_{h,\tau}(x,t)),w\rangle &= \langle \mathcal{R}(u_{h,\tau}(x,t)-\baruhtaulift(x,t)),w\rangle + \langle \mathcal{R}(\baruhtaulift(x,t)),w\rangle \nonumber \\
	&= \langle \Rcoarse, w \rangle +  \langle \mathcal{R}(\baruhtaulift(x,t)),w\rangle.
\end{align}
We deal with the coarsening later and focus on the other term first. 

 We start by analysing the residual \eqref{eq:res} evaluated at $\baruhtaulift(x,t)$ and insert the full discretization \eqref{eq:fulldiscreteheat} with an appropriate test function $W^n$.  For later analysis this test function has to be defined very carefully.
First assume that for some $t \in (t^{n-1},t^n]$ we extended a general $w \in H^1(\Gamma(t))$ to a function $\underline{w}$ defined on the space--time slab 
$$Q_{t^{n-1},t^n} :=\bigcup_{s \in (t^{n-1},t^n]} \Ga(s) \times \{s\},$$
such that $\partialmh \underline{w} = 0$, i.e.\ extending constantly along the flow induced by $v_h$ similar to the extension of finite element functions in Section~\ref{ch:subsec:full_dis_time_int}. Notice that the extension and lift do not commute.
Forcing this property for the discrete flow instead of the continuous flow simplifies later bounds. The extension allows us to choose $(\underline{w}(t^n))^{-\ell^n}=:W^n \in H^1(\Gamma_h^n)$, which we use as a general test function to the full discrete problem \eqref{eq:fulldiscreteheat}.  

\subsubsection{Residual decomposition}
We start by decomposing the residual: first we add the full discrete equation but tested with a general function in $H^1(\Ga_h^n)$, and afterwards analyse the differences of each discrete form to its continous variant. Note that if the arguments of the bilinear forms are evaluated at the same time (i.e.\ not flown from another time), we suppress the additional time dependencies. This results for $w \in H^1(\Ga(t)$ in  
\begin{align}
	\label{eq:res_pm_fulldis}
	\langle \mathcal{R}(\baruhtaulift(x,t)),w\rangle &=  m_h ^n(\partial_h ^\bullet \baruhtau - F_h,W) + g_h ^n(V_h;U_h,W) + a_h ^n(U_h,W) \tag*{$(\mathcal{R}_h)$} \\
	&\hspace{1.5em}+m^t(\partialm \baruhtaulift,w) - m_h ^n(\partialmh  \baruhtau,W)  \tag*{$(I)$} \\
	&\hspace{1.5em}+g^t(v;\baruhtaulift,w) - g_h ^n(V_h;U_h,W) \tag*{$(II)$}\\
	&\hspace{1.5em}+a^t(\baruhtaulift,w) - a_h ^n(U_h,W) \tag*{$(III)$} \\
	&\hspace{1.5em}-m^t(f,w) + m_h^n(F_h,W), \tag*{$(IV)$}
\end{align}
where $W^n = (\underline{w}(t^n))^{-\ell^n}$.
This immediately gives the spatial residual $\mathcal{R}_h$ tested with $W^n$. 
Due to the time dependency and the difference in arguments, e.g.\, $v$ vs. $V_h$, we insert zeros and collect the resulting terms as follows
\begin{align*}
I \pm m^t(\partial_h ^\bullet \baruhtaulift,w)&= \underbrace{m^t((\partialm -\partialmh) \baruhtaulift,w)}_{I.v}+\underbrace{m^t(\partial_h ^\bullet \baruhtaulift,w)- m_h ^n(\partial_h ^\bullet \baruhtau,W)}_{I.a}, \\
II \pm g^t(v_h;\baruhtaulift,w) &=  \underbrace{g^t(v-v_h;\baruhtaulift,w)}_{II.v} + \underbrace{g ^t(v_h;\baruhtaulift,w) - g_h ^n(V_h;U_h,W)}_{II.a}.
\end{align*}
Next we extract the oscillation-type terms which result from discretizing the general function $f$ and the velocity $v$ thus
\begin{align*}
II.a \pm m^t((\underline{\nabla_{\Ga_h^n} \cdot V_h^n})^{\ell^n(t)} \baruhtaulift,w) &= \underbrace{m^t\Big(\big((\nabla_{\Ga} \cdot v_h)-(\underline{\nabla_{\Ga_h^n} \cdot V_h^n})^{\ell^n(t)} \big) \baruhtaulift,w\Big)}_{II.\text{osc}} \\
&+ \underbrace{m^t((\underline{\nabla_{\Ga_h^n} \cdot V_h^n})^{\ell^n(t)} \baruhtaulift,w) - g_h ^n(V_h;U_h,W)}_{II.b}, \\
IV \pm m^t((\underline{F_h^n})^{\ell^n(t)},w) &= \underbrace{m^t(f-(\underline{F_h^n})^{\ell^n(t)},w)}_{IV.\text{osc}} + \underbrace{m^t((\underline{F_h^n})^{\ell^n(t)},w) - m_h^n(F_h,W)}_{IV.a}.
\end{align*}
The next step is to move from the continuous domain at time $t$ to the discrete time domain at $t^n$, to establish this we introduce additional bilinear forms. Note that we often immediately obtain geometric error terms:
\begin{align*}
I.a \pm m^n(\partial_h ^\bullet \baruhtaulift,\underline{w}) &= \underbrace{m^t(\partial_h^\bullet \baruhtaulift,w) -m^n(\partial_h ^\bullet \baruhtaulift,\underline{w}) }_{I.\textnormal{move}}+\underbrace{m^n(\partial_h ^\bullet \baruhtaulift,\underline{w})- m_h ^n(\partial_h ^\bullet \baruhtau,W)}_{I.\text{geo}}, \\
II.b \pm m^n((\nabla_{\Ga_h^n} \cdot V_h^n)^{\ell^n} \baruhtaulift(\cdot,t,t^n),\underline{w}) &= \underbrace{m^t((\underline{\nabla_{\Ga_h^n} \cdot V_h^n})^{\ell^n(t)} \baruhtaulift(\cdot,t^n),w) -m^n((\nabla_{\Ga_h^n} \cdot V_h^n)^{\ell^n} \baruhtaulift(\cdot,t,t^n),\underline{w}) }_{II.\textnormal{move}} \\
\qquad &+\underbrace{m^n((\nabla_{\Ga_h^n} \cdot V_h^n)^{\ell^n} \baruhtaulift(\cdot,\cdot,t^n),\underline{w})- g_h ^n(V_h;U_h,W)}_{II.c}, \\
III \pm a^n(\baruhtaulift(\cdot,t,t^n),\underline{w}) &= \underbrace{a^t(\baruhtaulift(\cdot,t^n),w)-a^n(\baruhtaulift(\cdot,t,t^n),\underline{w})}_{III.\textnormal{move}}+\underbrace{a^n(\baruhtaulift(\cdot,t,t^n),\underline{w})- a_h ^n(U_h,W)}_{III.c}, \\
IV.a \pm m^n(f_h,\underline{w}) &= \underbrace{m^t((\underline{F_h^n})^{\ell^n(t)},w)  -m^n(f_h,\underline{w}) }_{IV.\textnormal{move}}+\underbrace{m^n(f_h,\underline{w})- m_h^n(F_h,W)}_{IV.\text{geo}}. 
\end{align*} 
Finally we have to resolve the time dependency in $II.c$ and $III.c$ which naturally introduces temporal-type errors
\begin{align*}
II.c \pm m^n((\nabla_{\Ga_h^n} \cdot V_h^n)^{\ell^n} u_h,\underline{w}) &= \underbrace{m^n((\nabla_{\Ga_h^n} \cdot V_h^n)^{\ell^n} (\baruhtaulift(\cdot,t,t^n) - u_h),\underline{w}) }_{\mathcal{R}^\star_\tau}+\underbrace{m^n((\nabla_{\Ga_h^n} \cdot V_h^n)^{\ell^n} u_h,\underline{w})- g_h ^n(V_h; U_h ,W)}_{II.\text{geo}}, \\
III.c \pm a^n(u_h,\underline{w}) &= \underbrace{a^n(\baruhtaulift(\cdot,t,t^n) - u_h,\underline{w})}_{\mathcal{R}_\tau}+\underbrace{a^n(u_h,\underline{w})- a_h ^n(U_h,W)}_{III.\text{geo}}.
\end{align*}

Before stating the residuals, we simplify the terms $I.v$ and $II.v$ by using $\partialm w -\partialmh w = (v-v_h) \cdot \nabla_{\Gamma} w$ which holds for $w \in H^1(\Ga(t))$ by \cite[Lemma 4.1]{Dziuk2013FiniteEM} and partial integration
\begin{align*}
	I.v+II.v &= m^t((v-v_h) \cdot \nabla_{\Ga(t)} \baruhtaulift + [\nabla_{\Ga(t)} \cdot (v-v_h)] \baruhtaulift,w)  \\
	 &= m^t (\nabla_{\Ga(t)} \cdot [(v-v_h) \baruhtaulift],w) = \underbrace{-m^t((v-v_h)\baruhtaulift,\nabla_{\Ga(t)}w)}_{\mathcal{R}_v},
\end{align*}

Using all of the above reformulations, we gather the residuals
\begin{equation}
\label{eq:res_decom}
\begin{aligned}
	\langle \mathcal{R}_h,w \rangle &= m_h ^n(\partial_h ^\bullet \baruhtau,W) + g_h ^n(V_h;U_h,W) + a_h ^n(U_h,W), \\
	\langle \mathcal{R}_\tau,w \rangle &=  a^n(\baruhtaulift(x,t,t^n)-u_h,\underline{w}), \\
	\langle \mathcal{R}^\star_\tau,w \rangle &=  m^n((\nabla_{\Ga_h^n} \cdot V_h^n)^{\ell^n} (\baruhtaulift(x,t,t^n) - u_h),\underline{w}), \\
	\langle \mathcal{R}_\text{geo},w \rangle &= m^n(\partial_h^\bullet \baruhtaulift - f_h,\underline{w}) - m_h^n(\partial_h^\bullet \baruhtau - F_h,\underline{W}) \\
	&\qquad+m^n((\nabla_{\Ga_h^n} \cdot V_h^n)^{\ell^n},u_h,\underline{w})-g_h ^n(V_h;U_h,W) +a^n(u_h,\underline{w})-a_h^n(U_h,W), \\
	\langle \mathcal{R}_\textnormal{move},w \rangle &= m^t(\partial_h^\bullet \baruhtaulift,w) -m^n(\partial_h ^\bullet \baruhtaulift,\underline{w}) +m^t((\underline{F_h^n})^{\ell^n(t)},w)  -m^n(f_h,\underline{w}) \\
	&\qquad + m^t((\underline{\nabla_{\Ga_h^n} \cdot V_h^n})^{\ell^n(t)} \baruhtaulift,w) -m^n((\nabla_{\Ga_h^n} \cdot V_h^n)^{\ell^n} \baruhtaulift(x,t,t^n),\underline{w}) \\
	 &\qquad + a^t(\baruhtaulift,w)-a^n(\baruhtaulift(x,t,t^n),\underline{w}),\\
	\langle \mathcal{R}_v ,w \rangle &= m^t(-(v-v_h) \baruhtaulift, \nabla_{\Ga(t)} w).
\end{aligned}
\end{equation}
Combined with the coarsening residual $\mathcal{R}_c$ from \eqref{eq:residual_coarsening_decom} and the oscillation $\langle \osc^n ,w \rangle$ from \eqref{eq:oscillation}, the above residuals sum up to the exact residual \eqref{eq:res}.

For stationary surfaces, these residuals coincide with the residuals in \cite[Section~5.2]{KL25} or vanish due to $v \equiv 0$, except that the temporal residual is not unlifted to the discrete domain as this is not necessary for the upcoming proofs.

\subsection{Relating the residual and indicators}
\label{sec:relating_residual_ind}
Due to the residual decomposition many terms, like the spatial indicator, are closely related to terms of the analysis in \cite{Verfuerth2003, KL25}, or to the geometric residual to \cite[Section~5.4]{KL25}. We highlight the main difficulties which arise for the evolving surface case.

Before bounding each residual separately, we introduce an additional tool to compare $L^2$- and $H^1$-norms for discrete and general times $t \in (t^{n-1},t]$. 
\begin{proposition}[Norm equivalence under movement {\cite[Lemma~3.6]{DzEll12_Fully_Disc_ESFEM}}]
\label{prop:norm_equiv_partialmh}
Given $f$ defined on the space--time slab $Q_{a,b}$ with $\partial_h^\bullet f = 0$, then, for $s,t \in (a,b)$,
\begin{equation*}
\label{eq:norm_equiv_partialmh}
\begin{aligned}
\frac{1}{c(\kappa)} \|f(s)\|^2_{L^2(\Ga(s))} &\leq \|f(t)\|^2_{L^2(\Ga(t))} &&\leq c(\kappa) \|f(s)\|^2_{L^2(\Ga(s))} , \\
\frac{1}{c(\kappa)} |f(s)|^2_{H^1(\Ga(s))}  &\leq  |f(t)|^2_{H^1(\Ga(t))} &&\leq c(\kappa) |f(s)|^2_{H^1(\Ga(s))} .
\end{aligned}
\end{equation*}
The constant depends exponentially on $\kappa|s-t|$.
\end{proposition}

\subsubsection{Spatial residuals}
\label{sec:res_spatial}
First, let us highlight that the spatial and the geometric residuals are fully time independent, thus the stationary analysis is almost directly viable, except that we additionally have to control the evolution of the surface, in particular the different evaluation times of the test function $w$ of the dual norm and the respective pushed-forward $W^n$ in the spatial residual.
\begin{proposition}
\label{prop:spatial_res}
	For $0 < t^n \leq T$ the spatial indicator $\eta_h^n$ \eqref{eq:indicator - spatial} is uniformly equivalent to the dual norm of the spatial residual $\mathcal{R}_h$ in \eqref{eq:res}, i.e., for $t \in (t^{n-1},t^n]$,
	\begin{subequations}
	\begin{align}
		\label{eq:indicator upper bound for R_h}
		\| \mathcal{R}_h(t) \|_{{H^{-1} (\Ga(t))}} \leq &\  c C_\kappa^n \, \eta_h^n , \\
		\label{eq:indicator lower bound for R_h}
		 \frac{c}{C_\kappa^n} \eta_h^n \leq &\ \| \mathcal{R}_h(t) \|_{{H^{-1} (\Ga(t))}}.
	\end{align}
	\end{subequations}
	The constant $c > 0$ is independent of $h$ and $\tau^n$, but depends on the shape-regularity constant $\varrho^n$ of $\Ga_h^n$ and the norm-equivalence constant \eqref{eq:norm_equiv} associated with $\Ga(t)$.
\end{proposition}
\begin{proof}
The proof follows \cite[Proposition~5.6]{KL25} which extended the results of \cite[Section~5]{Verfuerth2003} to surfaces. The key ideas are to introduce the Scott--Zhang interpolation \cite[Section~3]{Camacho2014} of the function $W^n$, and use standard estimates. The lower bound is shown using typical bubble functions arguments \cite{Verfuehrt1996}. The exponential dependence on $\kappa$ arises due to the resulting $|W^n|_{H^1(\Ga_h^n)}$ in the upper bounds which has to be related back to $\|w(t)\|_{H^1(\Ga(t))}$ via Proposition~\ref{prop:norm_equiv_partialmh} to obtain the dual norm. \hfill
\end{proof}

\subsubsection{Geometric residuals}
\label{sec:res_geo}
For some of the indicators, including the geometric residual, we will not be able to incorporate them in the efficiency analysis.
However, we will see that these terms are usually of higher order.
\begin{proposition}
\label{prop:geo_res}
	For $0 < t^n \leq T$ the geometric residuals $\mathcal{R}_{\text{geo}}$, and $\mathcal{R}_v$ are bounded from above by the geometric indicators $\mathcal{G}_h^n$ \eqref{eq:indicator - geo_h} with additional higher order contributions of \eqref{eq:indicator - spatial}, and by $\mathcal{G}_v^n$ \eqref{eq:indicator - geo_v} respectively.
	\begin{align}
		\label{eq:indicator upper bound for R_geo}
		\refstepcounter{equation}\tag{\theequation a}
		\| \mathcal{R}_{\text{geo}} \|_{L^2(t^{n-1},t^n;{H^{-1} (\Ga(\cdot))})}^2 \leq &\ c (C_\kappa^n)^2 \tau^n \, \big((\mathcal{G}_h^n)^2+h^2 (\eta_h^n)^2\big), \\
		\label{eq:indicator upper bound for R_v}
		\refstepcounter{equation}\tag{\theequation a}
		\| \mathcal{R}_{v} \|_{L^2(t^{n-1},t^n;{H^{-1} (\Ga(\cdot))})}^2 \leq &\ c (C_\kappa^n)^2 \tau^n \, (\mathcal{G}_v^n)^2.
	\end{align}
	The constant $c > 0$ is independent of $h$ and $\tau^n$, but depends on $Q_{t^{n-1},t^n}$.
\end{proposition}
\begin{proof}
The spatial geometric residual $\mathcal{R}_{\text{geo}}$ is directly bound using the bounds between geometric and continous bilinear forms, see Proposition~\ref{prop:geo_err_bilin_form}, and via the Cauchy--Schwarz inequality, where the mass terms $m$ and weighted mass terms $g$ are combined. Yielding
\begin{align*}
\langle \mathcal{R}_{\text{geo}}(t),w \rangle &\leq ch^2 \Big(\|\partial_h^\bullet \baruhtau - F_h^n+ (\nabla_{\Ga_h^n} \cdot V_h^n) U_h^n\|_{L^2(\Ga_h^n)} + | U_h^n  |_{H^1(\Ga_h^n)} \Big) \|W^n\|_{H^1(\Ga_h^n)}\\
&\leq c(h \eta_h^n + \mathcal{G}_h^n) \|W^n\|_{H^1(\Ga_h^n)}.
\end{align*}
Note that we estimated elementwise, to obtain the relation to the indicators.
Integration in time combined with Proposition~\ref{prop:norm_equiv_partialmh} to shift the evaluations of the test function to $t$, gives the desired bound. 
By integrating the time-dependent quantity of the time-shift, we obtain
\begin{align*}
\int_{t^{n-1}}^{t^n} \exp(\kappa^n |t^n-s|)\d s \leq \tau^n \exp(\kappa^n \tau^n) = \tau^n C_\kappa^n.
\end{align*}
%

The geometric velocity residuals are bound similarly in a direct fashion:
\begin{align*}
\langle \mathcal{R}_{v}(t), w \rangle &\leq \|v-v_h\|_{L^\infty(\Ga(t))} \|\baruhtaulift\|_{L^2(\Ga(t))} |w|_{H^1(\Ga(t))} \leq c h^2 \|\baruhtaulift\|_{L^2(\Ga(t))} |w|_{H^1(\Ga(t))} \\
&\leq c h^2 \|\baruhtau\|_{L^2(\Ga_h^n(t))} \|w\|_{H^1(\Ga(t))} \leq c C_\kappa^n \mathcal{G}_v^n \|w\|_{H^1(\Ga(t))}.
\end{align*}
By \cite[Lemma~7.2]{Lubich_Mansour2015} we bound the difference $v-v_h$, the approximation $\baruhtaulift$ is directly bound by triangle inequality combined with Proposition~~\ref{prop:norm_equiv_partialmh} and time integration. \hfill
\end{proof}

\subsubsection{Movement residual}
\label{sec:res_move}
The movement yields a novel residual which requires a different approach then the usual indicators as we have to compare quantities on two states of the surface. By construction the arguments in the bilinear forms are fixed along the flow such that the movement residual only measures the perturbation resulting from the transport of these quantities.
\begin{proposition}
\label{prop:move_res}
	For $0 < t^n \leq T$ the movement residual $\mathcal{R}_{\textnormal{move}}$ is bounded from above by the movement indicator $\zeta_\textnormal{move}^n$ \eqref{eq:indicator - move} and the temporal indicators $\eta_\tau^n$ \eqref{eq:indicator - temporal}, $\zeta_\tau^n$ \eqref{eq:indicator - temporal2}.
	\begin{align}
		\label{eq:indicator upper bound for R_move}
		\refstepcounter{equation}\tag{\theequation a}
		\| \mathcal{R}_{\textnormal{move}} \|_{L^2(t^{n-1},t^n;{H^{-1} (\Ga(t))})}^2 \leq &\ c (C_\kappa^n)^2 \tau^n \big((\zeta_\textnormal{move}^n)^2+(\kappa^n \tau^n)^2 (\eta_\tau^n + \zeta_\tau^n)^2\big).
	\end{align}
	The constant $c > 0$ is independent of $h$ and $\tau^n$, but depends on $Q_{t^{n-1},t^n}$ via \eqref{eq:norm_equiv}.
\end{proposition}
Note that for non-moving domains all contributions of the moving residual would collapse to 0.
\begin{proof}
We restate the movement residual from \eqref{eq:res_decom} with explicit time dependencies
\begin{align*}
\langle \mathcal{R}_\textnormal{move},w \rangle = &m^t(\partial_h^\bullet \baruhtaulift(x,t,t),w) -m^n(\partial_h^\bullet \baruhtaulift(x,t,t^n),\underline{w}) + m^t((\underline{F_h^n})^{\ell^n(t)},w)  -m^n((F_h^n)^{\ell^n},\underline{w})\\
&+ m^t((\underline{\nabla_{\Ga_h^n} \cdot V_h^n})^{\ell^n(t)} \baruhtaulift(x,t,t),w) -m^n((\nabla_{\Ga_h^n} \cdot V_h^n)^{\ell^n} \baruhtaulift(x,t,t^n),\underline{w})\\
	 &+ a^t(\baruhtaulift(x,t,t),w)-a^n(\baruhtaulift(x,t,t^n),\underline{w}).
\end{align*}
Observe that all arguments are the same functions but once represented at time $t$ and once at time $n$ thus we rewrite the terms using the fundamental theorem of calculus:
\begin{align*}
\int_{t^n}^t &\frac{\d}{\d s} \Big[m^s(\partial_h^\bullet \baruhtaulift(x,t,s) ,\underline{w})- m^s((\underline{F_h^n})^{\ell^n(s)} ,\underline{w}) + m^s ((\underline{\nabla_{\Ga_h^n} \cdot V_h^n})^{\ell^n(s)},\baruhtaulift(x,t,s),\underline{w}) +a^s(\baruhtaulift(x,t,s),\underline{w})  \Big] ds.
\end{align*}
Next we insert $\pm (\underline{U_h^n})^{\ell^n(s)}$ into the bilinear forms containing the $t$-dependent $\baruhtaulift(x,t,s)$, which is done to step by step construct an indicator closely related to the spatial indicator. We collect these terms and handle them separately. Note that $\partial_h^\bullet \baruhtaulift(x,t,s)$ is independent of $t \in (t^{n-1},t^n]$ interval-wise:
\begin{align*}
A &:= \frac{\d}{\d s} \Big[m^s(\partial_h^\bullet \baruhtaulift(x,t,s) - (\underline{F_h^n})^{\ell^n(s)} + (\underline{\nabla_{\Ga_h^n} \cdot V_h^n})^{\ell^n(s)}(\underline{U_h^n})^{\ell^n(s)}  ,\underline{w}) +a^s((\underline{U_h^n})^{\ell^n(s)},\underline{w}) \Big], \\
B &:= \frac{\d}{\d s} \Big[m^s (\underline{\nabla_{\Ga_h^n} \cdot V_h^n})^{\ell^n(s)} \big(\baruhtaulift(x,t,s)-(\underline{U_h^n})^{\ell^n(s)}\big),\underline{w})+a^s(\baruhtaulift(x,t,s)-(\underline{U_h^n})^{\ell^n(s)},\underline{w})\Big].
\end{align*}

We differentiate these terms, where it is crucial to note that the material derivative of all arguments vanish by the construction of our extensions.
Thus the only contributions resulting from $A$ are the ones arising due to the evolving surface Leibniz rule \eqref{eq:leibniz_formula}:
\begin{align*}
A = g^s(v_h; \partial_h^\bullet \baruhtaulift(x,t,s) - (\underline{F_h^n})^{\ell^n(s)} + (\underline{\nabla_{\Ga_h^n} \cdot V_h^n})^{\ell^n(s)}(\underline{U_h^n})^{\ell^n(s)},\underline{w}) + b^s(v_h;(\underline{U_h^n})^{\ell^n(s)},\underline{w}).
\end{align*}

Now we can directly estimate the terms, using Proposition~\ref{prop:norm_equiv_partialmh} and norm equivalence of the lift \eqref{eq:norm_equiv} to arrive at
\begin{align*}
A \leq c C_\kappa^n \kappa^n \Big(\Big\|\frac{1}{\tau^n} (U_h^n-\IntRef U_h^{n-1}) + (\nabla_{\Ga^n} \cdot V_h^n) U_h^n -F_h^n\Big\|_{L^2(\Ga_h^n)} +|U_h^n|_{H^1(\Ga_h^n)} \Big) \|w\|_{H^1(\Gamma(t))}.
\end{align*}
Time integration yields the bound:
\begin{align*}
\int_{t^n}^t A &\leq c C_\kappa^n \kappa^n \tau^n \Big( \Big\|\frac{1}{\tau^n}(U_h^n-\IntRef U_h^{n-1}) + (\nabla_{\Ga^n} \cdot V_h^n) U_h^n -F_h^n\Big\|_{L^2(\Ga_h^n)} + |U_h^n|_{H^1(\Ga_h^n)} \Big) \|w\|_{H^1(\Gamma(t))} \\
&\leq C_\kappa^n \zeta_\textnormal{move}^n \|w\|_{H^1(\Gamma(t))}.
\end{align*}

For the second term we write, using the definitions of the time interpolation $\baruhtaulift(x,t,s)$ from \eqref{eq:fully discrete solution definition},
\begin{align*}
\baruhtaulift(x,t,s)-(\underline{U_h^n})^{\ell^n(s)} = \Big(\frac{t-t^{n-1}}{\tau^n}-1\Big) (\underline{U_h^n}-\underline{\IntRef U_h^{n-1}})^{\ell(s)} := z(\cdot,t,s).
\end{align*}

Again with respect to differentiation in s the material derivative vanishes $\partialmh |_s z = 0$ thus similar to the prior result we have
\begin{align*}
B = g^s(v_h; (\underline{\nabla_{\Ga_h^n} \cdot V_h^n})^{\ell^n(s)} z,\underline{w}) + b^s(v_h; z,\underline{w}).
\end{align*}
We integrate, flipping the integration bounds and bounding via norm equivalence of the lift \eqref{eq:norm_equiv} and under movement \eqref{eq:norm_equiv_partialmh}, and using the velocity bound \eqref{eq:kappa},
\begin{align*}
\int_{t^n}^t B \d s &= -\Big(\frac{t-t^{n-1}}{\tau^n}-1\Big) \int_{t}^{t^n} g^s(v_h; (\underline{\nabla_{\Ga_h^n} \cdot V_h^n})^{\ell^n(t)} (\underline{U_h^n}-\underline{\IntRef U_h^{n-1}})^{\ell(s)} ,\underline{w})\\
    &\hphantom{= -\Big(\frac{t-t^{n-1}}{\tau^n}-1\Big) \int_{t}^{t^n}} + b^s(v_h; (\underline{U_h^n}-\underline{\IntRef U_h^{n-1}})^{\ell(s)},\underline{w}) \d s \\
&\leq \Big(1-\frac{t-t^{n-1}}{\tau^n}\Big) c \kappa^n \int_{t}^{t^n} \|(\underline{\nabla_{\Ga_h^n} \cdot V_h^n}) (\underline{U_h^n}-\underline{\IntRef U_h^{n-1}})\|_{L^2(\Gamma_h^n(s))} \|\underline{w}\|_{L^2(\Gamma(s))}  \\
&\hphantom{\leq \Big(1-\frac{t-t^{n-1}}{\tau^n}\Big) \kappa^n \int_{t}^{t^n}} +  |\underline{U_h^n}-\underline{\IntRef U_h^{n-1}}|_{H^1(\Gamma_h^n(s))} |\underline{w}|_{H^1(\Gamma(s))} \d s \\
&\leq c C_\kappa^n \kappa^n \tau^n \Big( \|(\nabla_{\Ga_h^n} \cdot V_h^n) (U_h^n-\IntRef U_h^{n-1})\|_{L^2(\Gamma_h^n)} + |U_h^n-\IntRef U_h^{n-1}|_{H^1(\Gamma_h^n)} \Big) \|w\|_{H^1(\Gamma(t))} \\
&\leq c C_\kappa^n \kappa^n \tau^n  (\eta_\tau^n+\zeta_\tau^n) \|w\|_{H^1(\Gamma(t))}.
\end{align*}
Taking the $L^2$-norm of both terms in time yields the result. \hfill
\end{proof}

\subsubsection{Coarsening on moving domains}
\label{sec:res_coarse}
The coarsening on surfaces is non-trivial but tools to tackle these terms were introduced in \cite[Lemma~5.5]{KL25}. We extend these results to moving surfaces and show the following upper bound:
\begin{proposition}
\label{prop:coarse_res}
	For $0 < t^n \leq T$ and a surface $\Ga(t) \in C^2$ where $\Ga_h^n$ is an admissible triangulation for $h \leq h_0$, which is obtained using NVB refinement and based on the constructions described in Section~\ref{sec:smallest_common_triangulation}.
	Then, there is an $h$- and $n$-uniform constant $c > 0$, depending only on $\Ga(t)$, such that the coarsening residual from \eqref{eq:residual_coarsening_decom} is bounded by the coarsening indicator \eqref{eq:indicator - coarse} and mesh-transfer indicator \eqref{eq:indicator - refinement} as
	\begin{equation} 
	\label{eq:indicator upper bound coarse}
	\refstepcounter{equation}\tag{\theequation a}
		\|\Rcoarse(t)\|_{H^{-1} (\Ga(t))} \leq c C_\kappa^n \, (\eta_\textnormal{c}^n + \eta_\textnormal{trans}^n) .
	\end{equation}
\end{proposition}
\begin{proof}
First recall that the coarsening residual, from \eqref{eq:residual_coarsening_decom} explicitly reads:
\begin{equation*}
	\label{eq:res_coarse}
	\langle \mathcal{R}_c,w\rangle = m^t(\partial^\bullet (u_{h,\tau}-\baruhtaulift),w)+a^t(u_{h,\tau}-\baruhtaulift,w)+g^t(v;u_{h,\tau}-\baruhtaulift,w).
\end{equation*}
We simplify the difference, by the two definitions of the time interpolations \eqref{eq:lifted_discrete_sol} and \eqref{eq:fully discrete solution definition}, to
\begin{align*}
	u_{h,\tau}(x,t)-\baruhtaulift(x,t) &= \frac{t^n-t}{\tau^n}\big((\underline{U_h^{n-1}})^{\ell^{n-1}(t)} - (\underline{\IntRef U_h^{n-1}})^{\ell^{n}(t)}\big), \\
	\partial_h^\bullet (u_{h,\tau}(x,t)-\baruhtaulift(x,t)) &= \frac{1}{\tau^n}\big((\underline{\IntRef U_h^{n-1}})^{\ell^{n}(t)}-(\underline{U_h^{n-1}})^{\ell^{n-1}(t)} \big).
\end{align*}
As for the initial decomposition (see Section~\ref{sec:decomposition}) we try to swap to the discrete material derivative $\partialmh$ in the first term and the discrete velocity $v_h$ for the last. Resulting in
\begin{align*}
	\langle \mathcal{R}_c,w\rangle &= m^t(\partial_h^\bullet (u_{h,\tau}-\baruhtaulift),w)+g^t(v_h;u_{h,\tau}-\baruhtaulift,w) \\
	&+m^t((\partial^\bullet-\partial_h^\bullet) (u_{h,\tau}-\baruhtaulift),w)+g^t(v-v_h;u_{h,\tau}-\baruhtaulift,w)+a^t(u_{h,\tau}-\baruhtaulift,w).
\end{align*}
To apply the theory of \cite[Lemma~5.5]{KL25} we directly estimate everything and apply standard bounds between discrete and continuous quantities, this gives
\begin{align*}
	|\langle \mathcal{R}_c,w\rangle| &\leq c \Big(\big((\tau^n)^{-1}+\kappa^n+h^2\big) \big\|(\underline{\IntRef U_h^{n-1}})^{\ell^{n}(t)}-(\underline{U_h^{n-1}})^{\ell^{n-1}(t)}\big\|_{L^2(\Gamma(t))} \\
	 &+ (1+h^2)  \big| (\underline{U_h^{n-1}})^{\ell^{n-1}(t)} - (\underline{\IntRef U_h^{n-1}})^{\ell^{n}(t)}\big|_{H^1(\Gamma(t))} \Big) \|w\|_{H^1(\Gamma(t))}.
\end{align*}

It is important to observe that the difference between $\underline{U_h^{n-1}}$ and $\underline{\IntRef U_h^{n-1}}$ is only non-zero for elements which are coarsened directly before, or refined during, solving the discrete problem \eqref{eq:fulldiscreteheat} at time $t^n$. Thus only elements which change between timesteps are relevant, we will collect these elements in the set of refined $\mathcal{T}_{h,\textnormal{ref}}^{n-1}$ triangles (based on $\Ga_h^{n-1}$) and coarsened $\mathcal{T}_{h,\textnormal{coarse}}^{n}$ triangles (based on $\Ga_h^{n}$). For the upcoming analysis we show all results on the whole domain, but the indicator is represented and only computed on the elements which contribute.

We focus on the $L^2$-error first, we insert zeros based on representing both $\underline{U_h^{n-1}}$ and $\underline{\IntRef U_h^{n-1}}$ on the lifted common mesh:
\begin{align*}
	\| (\underline{\IntRef U_h^{n-1}})^{\ell^{n}(t)}- (\underline{U_h^{n-1}})^{\ell^{n-1}(t)} \| \leq &\ \|(\mathcal{I}_n^{\comesh} \underline{\IntRef U^{n-1}_h})^{\ell^{\comesh}(t)} - (\mathcal{I}_{n-1}^{\comesh} \underline{U^{n-1}_h})^{\ell^{\comesh}(t)} \| \\
	&\ + \| (\mathcal{I}_{n}^{\comesh} \underline{\IntRef U^{n-1}_h})^{\ell^{\comesh}(t)} - (\underline{\IntRef U^{n-1}_h})^{\ell^{n}(t)} \| \\
	&\ + \| (\mathcal{I}_{n-1}^{\comesh} \underline{U^{n-1}_h})^{\ell^{\comesh}(t)} - (\underline{U^{n-1}_h})^{\ell^{n-1}(t)} \|.
\end{align*}
We arrive at a typical coarsening term and two interpolation-type errors. The first term is unlifted to the discrete domain via norm equivalence, and further bound by time independent quantities via Proposition~\ref{prop:norm_equiv_partialmh}:
\begin{align*}
	\|(\mathcal{I}_n^{\comesh} \underline{\IntRef U^{n-1}_h})^{\ell^{\comesh}(t)} - &(\mathcal{I}_{n-1}^{\comesh} \underline{U^{n-1}_h})^{\ell^{\comesh}(t)} \|_{L^2(\Ga(t))} \leq c \|\mathcal{I}_n^{\comesh} \underline{\IntRef U^{n-1}_h} -\mathcal{I}_{n-1}^{\comesh} \underline{U^{n-1}_h}\|_{L^2(\Ga_h^{\comesh}(t))}   \\
	&\leq c C_\kappa^n \|\underline{\mathcal{I}_n^{\comesh} \underline{\IntRef U^{n-1}_h} - \mathcal{I}_{n-1}^{\comesh} \underline{U^{n-1}_h}}\|_{L^2(\Ga_h^{\comesh}(t^{n}))}.
\end{align*}
The two interpolation errors can be viewed as mesh-transfer errors, which are introduced by the non-linear lifting operator. Where the first of the two, related to $\underline{\IntRef U^{n-1}_h}$, measures the mesh transfer error resulting from coarsening, whereas the second, related to $\underline{U^{n-1}_h}$, measures the mesh transfer error resulting from refining. In fact, following the upcoming argument it is clear that the mesh transfers do not contribute to the error if the parent mesh (i.e.\ $\Ga_h^{n-1}(t)$ and $\Ga_h^n(t)$) align with the common mesh $\Ga_h^{\comesh}(t)$.

 In general, to analyse the mesh-transfer errors we require refinement methods which have a hierarchical structure like newest-vertex bisection. If we fix $t$ then we can follow similar arguments as in \cite[Lemma~5.5]{KL25}. The strategies still apply, however the nodal positions in general need an additional analysis step as the node correspondence does not correspond to a singular lift anymore, but a compositions of the flow and lifts at possibly different discrete timesteps.
We show the interpolation error pointwise in time, based on a $\theta$ argument similar to the one in \cite{KLLP_2017}.
For the last interpolation error this yields
\begin{align}
\label{eq:coarse_proof_int_bound}
	\| (\mathcal{I}_{n-1}^{\comesh} \underline{U^{n-1}_h})^{\ell^{\comesh}(t)} - (\underline{U^{n-1}_h})^{\ell^{n-1}(t)} \|_{L^2(\Ga(t))} &\leq ch^2 |\underline{U^{n-1}_h}|_{H^1(\Gamma_h^{n-1}(t))} \leq c C_\kappa^n h^2 | U^{n-1}_h|_{H^1(\Gamma_h^{n-1})}. 
\end{align}

The final inequality was again based on Proposition~\ref{prop:norm_equiv_partialmh}. 
The key idea for the first inequality is to explicitly describe the elementwise affine transformation between the meshes $\Gamma_h^{n-1}(t)$ and $\Gamma_h^{\comesh}(t)$. For the non-evolving case the transformation was based on a refined but not lifted version of $\Gamma_h^{n-1}(t)$ (i.e.\ tracking all refinement as the midpoint of an edge based on newest-vertex bisection) and then lifting refined nodes to the exact surface which resulted in $\Gamma_h^{\comesh}(t)$, atleast for a single refinement-level.
By construction the evaluation of the nodes for the refined but not lifted version of $\Gamma_h^{n-1}(t)$ and $\Gamma_h^{\comesh}(t)$ are equivalent. Thus the difference in the interpolation error could be represented as the difference of two identical finite element functions but with nodes perturbed under an elementwise affine transformation. 
In our setting if we lift the midpoints of the refined but not lifted mesh $\Gamma_h^{n-1}(t)$ at any time other than $t^n$ they will in general not match up with the corresponding node in $\Gamma_h^{\comesh}(t)$ (see Figure~\ref{fig:common_triangulation}) as this was possibly constructed by lifting at time $t^n$ and then pull-back to time $t$ under the exact flow. However, the elementwise affine transformation can also be constructed such that the refined but not lifted node is mapped linearly to the corresponding node in $\Gamma_h^{\comesh}(t)$. Then again the nodal values always match up and we are able to rewrite the interpolation difference using the fundamental theorem of calculus. 
We define the affine mapping based on the $\theta$-dependent nodal values $x_j^\theta(t) := x_j^0(t) + \theta (x_j^1(t) - x_j^0(t))$, where $x_j^0(t)$ are the refined but not yet lifted vertices of $\Gamma_h^{n-1}(t)$, and $x_j^1(t)$ are the vertices of $\Gamma_h^{\comesh}(t)$.
Based on these nodal values and mesh connectivity of the common triangulation we define the $\theta$-dependent domain as $\Gamma_h^\theta(t)$.
 In comparison to the stationary case we lost the property, for single level refinements, that $x_j^1 = (x_j^0)^{\ell(t)}$. 

It is further important to note that due to the misaligned movement of the node against the unlift of the point $a$ onto the two different discrete meshes, a node can be unlifted into two different elements (in any child of a parent element of $\Gamma_h^{n-1}(t)$), as illustrated in Figure~\ref{fig:mismatch_common_triangulation}, which showcases the need for patchwise estimators discussed in Remark~\ref{rem:semi_local_common_mesh} if the construction is not carefully done.

\begin{figure}[htb!]
\centering
\includegraphics[width=0.85\linewidth]{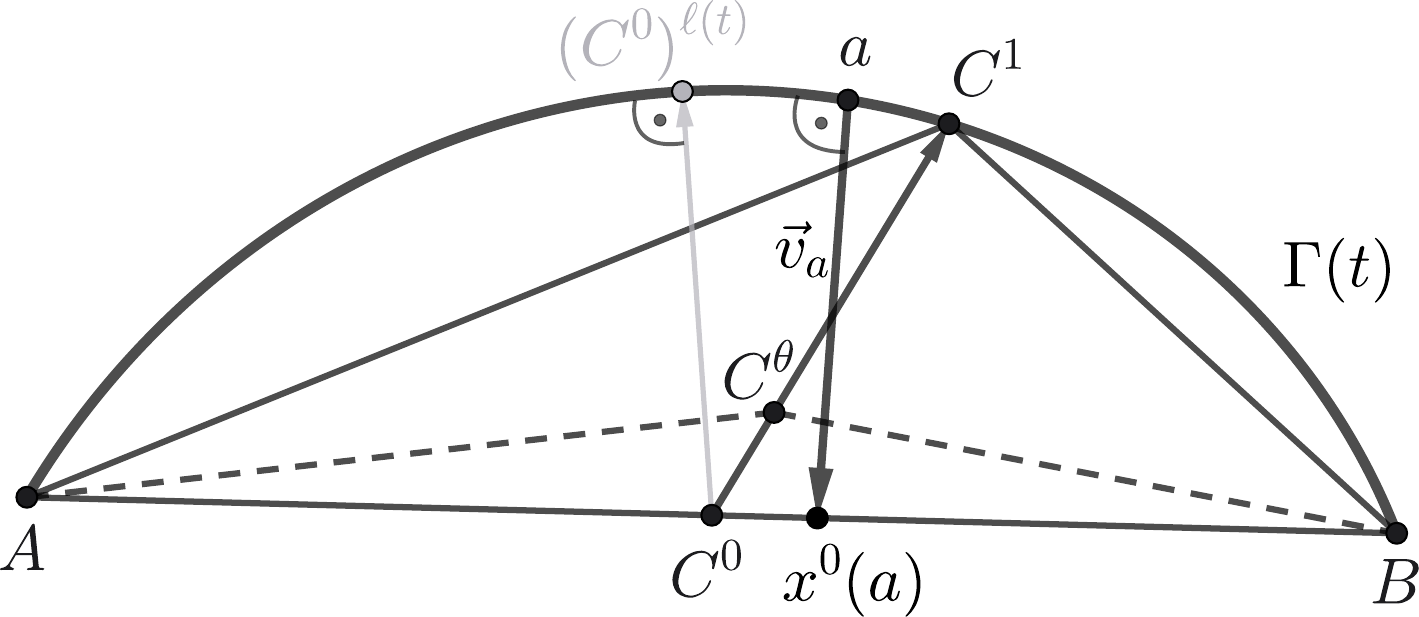}
\caption{Illustration of a two-dimensional $\theta$-dependent triangulation $\Gamma_h^\theta(t)$, where the refined but not lifted node $C^0$ is mapped onto $C^1$. The lifted node $(C^0)^{\ell(t)}$ shows the possible mismatch. Further the unlift of a general point $a \in \Ga(t)$ is given based on the normal $v_a$ of $\Gamma(t)$ which shows for different $\theta$ that the point unlifts to the element $AC^\theta$ for large $\theta$ but to the element $C^\theta B$ for small $\theta$.}
\label{fig:mismatch_common_triangulation}
\end{figure}

We describe the function $w_h^\theta \in S_h(\Gamma_h^\theta(t))$ as in the stationary case \cite[Lemma~5.5]{KL25} but with the nodes as discussed above. Thus the linear transformation is purely defined by the change in node positions and all nodal values are fixed.
If we identify $w_h^0 := \underline{U_h^{n-1}}$ and then following our argument $w_h^1 := \mathcal{I}_{n-1}^{\comesh} \underline{U_h^{n-1}}$ we can rewrite
\begin{align}
\label{eq:coarse_proof_fdm_thm}
(\mathcal{I}_{n-1}^{\comesh} \underline{U_h^{n-1}})^{\ell^{\comesh}} - (\underline{U_h^{n-1}})^{\ell^{n-1}(t)} =   (w_h^1)^{\ell^1} -(w_h^0)^{\ell^0} = \int_0^1 \frac{\d}{\d\theta} (w_h^\theta)^{\ell^\theta} \d\theta.
\end{align}
Similar to \cite[Lemma~5.5]{KL25}, one can explicitly differentiate and bound each contribution separately,
\begin{align*}
\frac{\d}{\d\theta} (w_h^\theta)^{\ell^\theta} = (\nabla_{\Gamma_h^\theta} w_h^\theta)^{\ell^\theta} \cdot \partial_\theta x^\theta(a) + (\partial_\theta w_h^\theta)(x^\theta(a)) 
\end{align*}
Some bounds follow immediately but for others additional analysis is required. In particular $\partial_\theta x^\theta(a)$, where $x^\theta \in \Ga_h^\theta$ is the unlift of some point $a$ onto the intermediate triangulation $\Ga_h^\theta$. Other than the possible element change in the common triangulation 
the $\theta$-dependent movement of the point $x^\theta$ can be described as determining the intersection of two lines, both moving linearly in terms of $\theta$ and with identical asymptotic velocities. Note that the nodal velocities are determined by the definition of the affine transformation of the vertices. For completeness we argue that our construction fulfils $x_j^1(t) - x_j^0(t) = \mathcal{O}(h^2)$ verifying the velocity order. 

For illustration we base the following argument for a 2D element, however the same holds for the surface setting. Many of the following quantities are also presented in Figure~\ref{fig:mismatch_common_triangulation}.

Given two nodes $A,B \in \Gamma_h^0$ of an element marked by NVB, with the resulting refined but not lifted node $C^0 := \frac{A+B}{2}$. Assume that this node is later required for refinement at time $t^n$. Then following the construction in Section~\ref{sec:smallest_common_triangulation} we first determine the midpoint at time $t^n$ which is given by $M := \frac{A^n+B^n}{2}$, where $A^n := G(A,t,t^n)$ and $B^n := G(B,t,t^n)$ (marked by a cross in Figure~\ref{fig:common_triangulation} after refinement). Next we lift it to $\Ga(t^n)$ resulting in $M^{\ell^n}$ which yield the vertices at time $\Ga_h^n$.
Then to arrive with the nodal position of the common triangulation we pull it back to time $t$ resulting in $C^1 := G(M^{\ell^n},t^n,t)$.

To compare $C^1-C^0$ we introduce the following extension of the flow map onto the tubular region $\mathcal{U}_\epsilon(t) \supset \Ga(t)$. By the assumptions on $\mathcal{U}_\epsilon(t)$ and the regularity of $\Ga(t)$ in Section~\ref{ch:surfacePDE} we can define the extended flow map for any $x \in \mathcal{U}_\epsilon(t)$ by $\tilde{G}(x,t,s) := G(x^{\ell(t)},t,s)$ via the closest point projection \eqref{eq:uniquedecom}. The regularity of $\tilde{G}$ is governed by the regularity of $G(\cdot,t,s) \in C^3$ in the spatial variable and the closest point projection, which is spatially $C^2$ by our assumption on $\Ga(t) \in C^3$. Thus $\tilde{G}(\cdot,t,s) \in C^2(\mathcal{U}_\epsilon(t))$.

Using the extension we can identify $G(M^{\ell^n},t^n,t) = \tilde{G}(M,t^n,t)$, which we expand by Taylor expansion in the first argument around the points $A^n$ and $B^n$
\begin{align*}
C^1  = \tilde{G}(\frac{A^n+B^n}{2},t^n,t) = &\frac{1}{2} \big(\tilde{G}(A^n,t^n,t) + D_x\tilde{G}(A^n,t^n,t)(A^n-M) + \frac{1}{2} D^2_x\tilde{G}(\xi_A,t^n,t)(A^n-M)^2 \big)  \\
+ &\frac{1}{2} \big(\tilde{G}(B^n,t^n,t) + D_x\tilde{G}(B^n,t^n,t)(B^n-M) + \frac{1}{2} D^2_x\tilde{G}(\xi_B,t^n,t)(B^n-M)^2 \big).
\end{align*}
Where $\xi_A,\xi_B \in [A^n,B^n]$ are points on the line segment between $A^n$ and $B^n$. Due to our construction we simplify by $A^n-M = \frac{A^n-B^n}{2} = -(B^n-M)$, use the bound $A^n-B^n = \mathcal{O}(h)$, the identity $\tilde{G}(x,t^n,t) = G(x,t^n,t)$ for $x \in \Ga(t)$, and the regularity of $\tilde{G}$ 
\begin{align*}
C^1 = \frac{G(A^n,t^n,t) + G(B^n,t^n,t)}{2} + c h (D_x\tilde{G}(A^n,t^n,t)-D_x\tilde{G}(B^n,t^n,t)) + \mathcal{O}(h^2).
\end{align*}
with a constant $c >0$.
But $G$ is a diffeomorphism, thus $G(A^n,t^n,t) = A$ and $G(B^n,t^n,t) = B$. If we additionally apply the mean-value-theorem on $D_x\tilde{G}$, which is applicable as the linear path $A^n$ to $B^n$ is contained in $\mathcal{U}_\epsilon(t)$ by the assumption on $h_0$ (see Remark~\ref{remark:h0 sufficiently small}), we can rearrange and simplify the inequality to
\begin{align*}
C^1-C^0 \leq  c h (A^n-B^n) + \mathcal{O}(h^2) \leq ch^2.
\end{align*}
which concludes that the difference of the intermediate node positions is quadratic and thus its corresponding $\theta$-velocity, i.e.\, $\partial_\theta ((1-\theta)C^0 + \theta C^1) = \mathcal{O}(h^2)$.

Following similar computations as in \cite{KL25} this implies elementwise $\partial_\theta x^\theta(a) = \mathcal{O}(h_T^2)$ with $h_T$ being the local element size of the coarser triangulation (i.e.\ $\Ga_h^{n-1}(t)$ or $\Ga_h^{n}(t)$ for the second interpolation error respectively). 
Now combining the bounds for $\partial_\theta x^\theta(a)$ and bounds for the affine transformation \cite[Eq.~(5.12)]{KL25} we can apply Minkowski's integral inequality and estimate the derivative terms of \eqref{eq:coarse_proof_fdm_thm}. This yields the first inequality of \eqref{eq:coarse_proof_int_bound} which is then represented on a computable mesh via norm-equivalence arguments. Summation over all elements gives the bound by $\eta_\textnormal{trans}^n$

The second interpolation-type bound for $\underline{\IntRef u_h^{n-1}}$ follows in the same manner, however only those elements, where vertices where coarsened contribute to this error. Again we represent the error for $\underline{\IntRef u_h^{n-1}}$ with the same arguments as in \eqref{eq:coarse_proof_int_bound} on $\Gamma_h^n$. Summation over all elements gives the bound by the second sum of $\eta_\textnormal{trans}^n$.


The $H^1$-seminorm estimate, follow the same structure as the $L^2$-error, and thus can be shown just as described in \cite[Lemma~5.5]{KL25}. The non-trivial element correspondence leads to the same-order bounds. 

\end{proof}

\subsubsection{Temporal residuals}
\label{sec:res_temp}
The reliability of the temporal residuals to their respective indicators can be shown with standard tools, however the efficiency proof is only shown up to oscillation, coarsening, mesh-transfer, high-order terms and the velocity-induced indicators.
\begin{proposition}
\label{prop:temp}
	For $0 < t^n \leq T$ the temporal residuals $\mathcal{R}_\tau$ and $\mathcal{R}^\star_\tau$ of \eqref{eq:res_decom} are bounded from above by $\eta_\tau^n$ and respectively $\zeta_\tau^n$ given in \eqref{eq:indicator - temporal}, \eqref{eq:indicator - temporal2}
	\begin{align}
		\label{eq:indicator upper bound for R_tau}
		\refstepcounter{equation}\tag{\theequation a}
		\|\mathcal{R}_\tau\|_{L^2(t^{n-1},t^n;H^{-1}(\Ga(\cdot)))}^2 \leq c_1 (C_\kappa^n)^2 \, \tau^n  \, (\eta_{\tau}^n)^2 . \\
		\refstepcounter{equation}\tag{\theequation a}
		\label{eq:indicator upper bound for R_taustar}
		\|\mathcal{R}_\tau^\star\|_{L^2(t^{n-1},t^n;H^{-1}(\Ga(\cdot)))}^2 \leq c_1 C_\kappa^n \, \tau^n  \, (\zeta_{\tau}^n)^2 .
	\end{align}

	Additionally, with $\tau \leq \tau_0$, where $\tau_0 >0$ is sufficiently small,  we have the local lower bound between temporal indicator $\eta_\tau^n$ \eqref{eq:indicator - temporal}  and the error $u-\baruhtaulift$,
	\begin{equation}
	\label{eq:final_temporal_bound_2}
	\refstepcounter{equation}\tag{\theequation b}
		\begin{aligned}
		(\tau^n)^{1/2} \eta_\tau^n \leq  c_2 C_\kappa^n \Big(\|u- \baruhtaulift&\|_{X(t^{n-1},t^n)} + \|\osc^n\|_{L^2(t^{n-1},t^n;H^{-1}(\Ga(\cdot)))} \\
		 &+ (\tau^n)^{1/2} (\zeta_\tau^n + \mathcal{G}_h^n + \mathcal{G}_v^n + \zeta_\textnormal{move}^n + \eta_\textnormal{c}^n+\eta_\textnormal{trans}^n)  \Big) .
		\end{aligned}
	\end{equation}		
	The constants $c_1,c_2 > 0$ are independent of $h$ and $\tau^n$, but depend on $Q_{t^{n-1},t^n}$, additionally $c_2$ depends on the shape-regularity constant $\varrho^n$ of $\Ga_h^n$.
\end{proposition}

\begin{proof}
We start with the upper bounds.
 First recall that, as in \cite[Eq.~6.3]{Verfuerth2003} we can rewrite the difference 
\begin{align*}
\baruhtaulift(\cdot,t,t^n)-u_h^n = \Big(\frac{t-t^{n-1}}{\tau^n}-1\Big)(u_h^n-\IntRef u_h^{n-1}).
\end{align*}
Thus we can directly bound using Cauchy--Schwarz inequality and lift equivalence \eqref{eq:norm_equiv}: 
\begin{align*}
|\langle \mathcal{R}_\tau , w \rangle| &\leq c |U_h^n-\IntRef U_h^{n-1} |_{H^1(\Ga_h^n)}  |\underline{w}|_{H^1(\Gamma_h^n(t))}, \\
|\langle \mathcal{R}^\star_\tau , w \rangle| &\leq c \|(\nabla_{\Gamma_h^n} \cdot V_h^n)(U_h^n-\IntRef U_h^{n-1}) \|_{L^2(\Ga_h^n)}  \|\underline{w}\|_{L^2(\Gamma_h^n(t))}.
\end{align*}
Determining the dual norm of the temporal residuals and then using norm equivalence under movement and time integration directly yields the indicators, i.e.\ \eqref{eq:indicator upper bound for R_tau} and \eqref{eq:indicator upper bound for R_taustar} with the $\kappa$-dependency as in Lemma~\ref{prop:norm_equiv_partialmh}.


The lower bound is again based on choosing the correct test function \cite[Section~7]{Verfuerth2003}. By choosing $w = \baruhtaulift(\cdot,t,t)-\underline{u_h^n}$ we immediately have
\begin{align*}
\int_{t^{n-1}} ^{t^n} \langle \mathcal{R}_\tau, w \rangle \d t= \int_{t^{n-1}} ^{t^n} |\underline{w}(t^n)|^2_{H^1(\Gamma_h^n)} \d t \geq c (\eta^n_\tau)^2 \int_{t^{n-1}} ^{t^n} \Big(\frac{t-t^{n-1}}{\tau^n}-1\Big)^2 \d t = c \frac{\tau^n}{3} (\eta^n_\tau)^2
\end{align*}
where we used the upper bound for the norm equivalence under lift \eqref{eq:norm_equiv}.

Now the lower bound is derived by first noting that the temporal residual $\mathcal{R}_\tau$ is invariant under constant shifts, that is $\langle \mathcal{R}_\tau,w\rangle =\langle \mathcal{R}_\tau,w+c\rangle$ for an arbitrary constant $c \in \R$. Thus we subtract the mean of $w$, labelled $\hat{w}:= \int_\Ga(t) w \in \R$ which allows us to use Poincarés inequality in the following argument:
\begin{align*}
\frac{\tau^n}{3} (\eta^n_\tau)^2 &\leq \int_{t^{n-1}} ^{t^n} \langle \mathcal{R}_\tau, w - \hat{w} \rangle \d t \\
&= \int_{t^{n-1}} ^{t^n} \langle \mathcal{R} - \mathcal{R}_h - \mathcal{R}_\text{geo} - \mathcal{R}_\tau^\star - \mathcal{R}_\textnormal{move} - \mathcal{R}_v - \osc^n - \mathcal{R}_c, w - \hat{w} \rangle \d t \\
&\leq c C_\kappa^n (1+C_P^2)^{1/2} \Big(\frac{\tau^n}{3}\Big)^{1/2} \eta_\tau^n \big(\|u-u_{h,\tau}\|_{X(t^{n-1},t^n)} + \|\osc \|_{L^2(t^{n-1},t^n;H^{-1}(\Ga(\cdot)))} \\
&\qquad \qquad \qquad + (\tau^n)^{1/2}(\eta_h^n  + \eta_\textnormal{c}^n +\eta_\textnormal{trans}^n + \mathcal{G}_h^n  + \mathcal{G}_v^n + \zeta_\tau^n + \zeta_\textnormal{move}^n) \big) 
\end{align*}
Which is based on the same ideas as in \cite[Section~5.7]{KL25} and employs the residual decomposition, a duality argument, and the derived dual norms of the prior chapters. Note that we used $h \leq h_0$ and $\tau \leq \tau_0$ to remove additional high-order terms, in particular geometric $h$-powers and the $\tau^2$ scaled terms of the movement indicator, in the arguments. 

As for the non-moving case, in the same style as \cite[Proposition~5.7]{KL25}, utilising the bound to the temporal indicator it is possible to bound the spatial indicator in terms of all other indicators but the temporal $\eta_\tau^n$, resulting in the bound
\begin{align}
\label{eq:final_spatial_bound}
	\refstepcounter{equation}\tag{\theequation b}
(\tau^n)^{1/2} \eta_h^n \leq c &C_\kappa^n \big(\|u-u_{h,\tau}\|_{X(t^{n-1},t^n)} + \|\osc^n \|_{L^2(t^{n-1},t^n;H^{-1}(\Ga))}  \\
&+ (\tau^n)^{1/2}(\zeta_\tau^n + \eta_\textnormal{c}^n + \eta_\textnormal{trans}^n + \mathcal{G}_h^n  + \mathcal{G}_v^n + \zeta_\textnormal{move}^n)\big). \nonumber
\end{align}
Note that the constant $c>0$ is independent of $h$ but depends on the constant from the norm equivalence. This shows that the spatial indicator is efficient up to the oscillation, high-order geometric terms, coarsening, mesh-transfer, and the indicators related to the movement of the surface $\zeta_\textnormal{move}^n$ and $\zeta_\tau^n$.

But then immediately the same follows for the temporal indicator. \hfill
\end{proof}

\begin{remark}
We note that the additional control of the velocity scaled temporal indicator yields an explicit tool to ensure that the mass transport at elements of strong local changes in the velocity only introduces small errors.
The efficiency, however, seems to be challenging as we can not test with a similar argument as the velocity divergence $\nabla_{\Ga_h^n} \cdot V_h^n$ is constant elementwise but discontinuous globally. It would be natural to try showing efficiency via bubble functions, however the scaling is off between efficiency and reliability. Additionally the uncontrollable sign of the argument is difficult to manage.
%
\end{remark}

\subsection{Proof of Theorem~\ref{thm:upper_lower}}
The proof of the upper (reliability) bound \eqref{eq:reliability estimate} begins with the estimate \eqref{eq:res_equiv_2}, which provides an upper bound of the error in terms of the residual. We continue by splitting the full residual via our decomposition \eqref{eq:res_decom} and combining the upper bounds of each term, i.e.\ the bound to the spatial residual \eqref{eq:indicator upper bound for R_h}, the bound to the geometric residuals \eqref{eq:indicator upper bound for R_geo} \& \eqref{eq:indicator upper bound for R_v}, the bound of the movement residual \eqref{eq:indicator upper bound for R_move}, the bound of the coarsening residual \eqref{eq:indicator upper bound coarse} and finally the bound of the temporal residuals \eqref{eq:indicator upper bound for R_tau} \& \eqref{eq:indicator upper bound for R_taustar}. Note that the oscillation \eqref{eq:oscillation} is split by triangle inequality but not further bound.

The proof of the lower (efficiency) bound \eqref{eq:efficiency estimate} begins with the estimate \eqref{eq:res_leq_error}, which provides a lower bound for the error in terms of the residual. We continue by combining the results for the spatial indicator \eqref{eq:indicator lower bound for R_h} \& \eqref{eq:final_spatial_bound} and for the temporal residual \eqref{eq:final_temporal_bound_2}, whilst reusing the upper bounds for the other indicators.

This finishes the proof of Theorem~\ref{thm:upper_lower}. \hfill \qedsymbol
\section{Numerical experiments}
\label{sec:numerical_experiments}

We investigate the derived error indicators \eqref{eq:indicator - full} and behaviour of a simple adaptive routine. The numerical experiments illustrate and complement the theoretical result. 
The implementation is based on the fully vectorized loop-free assembly $\ell$FEM package \cite{ellFEM}, which provides efficient computation of bulk and surface assembly. In particular it computes the mass, stiffness and velocity-scaled mass matrices in almost linear time-complexity.
All experiments use NVB for both refinement and coarsening (we modified the implementations of \cite{CoarsenNVB}), the bulk criterion \cite{Doerfler1996} for marking, and initial meshes are generated using DistMesh \cite{distmesh}.
We state a general structure for an adaptive algorithm, however note that no algorithmic results (like convergence) are given. The numerical experiments show the theoretically expected behaviour shown in Theorem~\ref{thm:upper_lower}.

%

\textbf{Algorithmic structure,}
We briefly outline a conceptual adaptive space--time algorithm illustrating how the proposed indicators can be employed.
Since the focus of this work is the a posteriori analysis rather than implementation aspects, we omit several algorithmic details and only describe the control logic.
The algorithm aims to ensure good convergence properties in terms of a tolerance $\TOL$ which bounds all indicators and thus relates to the residual and error by Theorem~\ref{thm:upper_lower}.
\begin{itemize}
\item Initialization: Determine the base mesh $\Ga_h^0$ (such that the closest point projection \eqref{eq:uniquedecom} exists for all times) and an initial mesh (such that the initial error is non-dominant)
\item For each timestep
\begin{enumerate}
\item Coarsen prior mesh as long as $\eta_\textnormal{c} \leq \TOL$
\item While $\zeta_\tau+\zeta_\textnormal{move} > \theta_\zeta \TOL$ reduce time-step size
\item Solve--Estimate--Mark--Refine until $\eta_h+\eta_g\leq \TOL$
\item Check temporal error $\eta_\tau \leq \theta_\tau \TOL$ and recheck $\zeta_\tau+\zeta_\textnormal{move} < \theta_\zeta \TOL$
\begin{enumerate}
\item If true, check mesh-transfer errors from coarsening and refinement $\eta_\textnormal{trans} \leq \theta_\textnormal{trans} \TOL$
\begin{enumerate}
\item If true, store solution and possibly increase time-step size for next step
\item If false, check what part of the transfer error dominates. If its the coarsening part, reduce coarsened elements, if it is the refinement part, return to previous step and solve on a finer mesh, which is used to restart at (1).
\end{enumerate}
\item If false, reduce time-step size and return to (2)
\end{enumerate}
\end{enumerate}
\end{itemize}
It is common in adaptive algorithms to allocate the total error budget unevenly among different indicators, by scaling individual tolerances. For simplicity we do not define a total tolerance and, mostly for illustrative purposes scale some of the indicators by  positive parameters $ \theta_\tau, \theta_\zeta, \theta_\textnormal{trans}$ to control the temporal, transport and mesh-transfer indicators, respectively.

In general, one can frame step (2) as a predictive transport control. Step (4) ensures that for a accepted temporal stepsize, based on $\eta_\tau$, not only the temporal estimator for the consistency of the PDE, but also the transport error, is small for the refined mesh. Additionally in Step (4) we check whether the mesh-transfer spoils the solution. If it does, depending on whether the coarsening transfer or the refinement transfer dominates, we would either reduce the number of coarsened elements in the current timestep, or respectively, refine additional elements in the previous timestep. 

 The spatial refinement loop is based on the ideas of \cite{Doerfler1996} combined with lifting just as described in \cite[Section~7]{KL25}.

We fixed some parameters for all upcoming proofs, namely if not specified otherwise $\theta_\tau = 1$, we always ensure that $\theta_\zeta = 0.75 \theta_\tau$, i.e. at most three-quaters of the temporal error budget, and we track if $\eta_\textnormal{trans}$ has a strong effect on the solution via $\theta_\textnormal{trans} = 5$.

Further the coarsening is implemented such that at most one refinement level can be coarsened within a single timestep. This could lead to slower coarsening, however strongly simplifies node tracking and the numerical experiments demonstrate that it is sufficient in practice. 
Also the increase and decrease of time-step size is fixed to always double or halve respectively.
Although the theoretical algorithm includes a correction mechanism based on the mesh-transfer indicator, it was not implemented in the numerical experiments. Instead, we monitor $\eta_\textnormal{trans}$ throughout all experiments, to check whether corrections would have been necessary. Except for the experiment in Section~\ref{sec:dumbell}, the transfer indicator remained non-dominant.

Finally we localized some of the indicators \eqref{eq:indicator - full}. In particular, rather than working with the spatially global $\kappa^n$ in the movement indicator \eqref{eq:indicator - move}, and coarsening indicator\eqref{eq:indicator - coarse}, we computed the divergence $\nabla_\Ga \cdot V_h^n$ and the tensor $\mathcal{B}_h^n$ elementwise, thereby obtaining a separate factor for each triangle. This refinement is consistent with the theory, but was omitted from the proof of Theorem~\ref{thm:upper_lower} for the sake of readability.

\subsection{Convergence: Bouncing Ellipsoid}
\label{sec:ex1}
First analyse the convergence with the help of a manufactured solution of the PDE \eqref{eq:heat_strong} on a bouncing ellipsoid (similar to \cite[Eq.~5.47]{Dziuk2013FiniteEM}) given by the implicit level-set function
\begin{align*}
d(x,y,z,t) = \frac{x^2}{1+0.9 \sin(2 \pi t)} + y^2 + z^2 -1.
\end{align*}
The surface starts initially as a sphere and periodically grows/shrinks along the $x$-axis. Given the exact solution $u = xy \exp(t)$, we determined the right-hand side $f$ based on \eqref{eq:heat_strong}. 
In Figure~\ref{fig:ex1} we observe the convergence rate of $\TOL$ with respect to the $L^\infty(L^2)$- and $L^2(H^1)$-error. The $L^2(H^1)$-errors were computed with a sufficiently high quadrature rule. As expected the left graph of Figure~\ref{fig:ex1} shows, by construction of the algorithm, that the errors behave asymptotically linear with respect to the tolerance.

Additionally, we investigate the coarsening behaviour. We observe for a specific solution with $\TOL = 1$, the right-hand side graph of Figure~\ref{fig:ex1}, exhibits an exponential node decay, structurally following the exponential decay in the solution $u$ itself.
We further observe that the exponential decay in degrees of freedom is affected by the surface evolution. During the initial expansion of the surface, the decay is slower since additional vertices are required to maintain good control of the error. Whereas after $t = 0.5$, as the surface shrinks, and the decay rate increases accordingly.

\begin{figure}
    \centering
    \begin{subfigure}[t]{0.48\linewidth}
        \centering
%
\begin{tikzpicture}

\begin{axis}[%
width=0.8\linewidth,
height=5cm,
at={(0cm,0cm)},
scale only axis,
xmode=log,
xmin=0.1,
xmax=1.25,
xminorticks=true,
xlabel style={font=\color{white!15!black}},
xlabel={Tolerance},
ymode=log,
ymin=0.01,
ymax=0.3,
yminorticks=true,
axis background/.style={fill=white},
legend style={
  at={(0.03,0.97)},
  anchor=north west,
  legend cell align=left,
  align=left,
  draw=white!15!black
},
]
\addplot [color=black, mark=o, mark options={solid, black}]
  table[row sep=crcr]{%
1.25	0.08602706637774\\
1	0.0846630081854134\\
0.75	0.0836309086588545\\
0.5	0.0663537477822094\\
0.3	0.0402140547602\\
0.2	0.0324469794426548\\
0.15	0.024640416353672\\
0.1	0.0150057430304665\\
};
\addlegendentry{$\|e_h\|_{L^\infty(L^2)}$}

\addplot [color=black, mark=asterisk, mark options={solid, black}]
  table[row sep=crcr]{%
1.25	0.140591113782617\\
1	0.125049735505677\\
0.75	0.106343742567335\\
0.5	0.0786120481647433\\
0.3	0.0474530880168746\\
0.2	0.036016283869202\\
0.15	0.0275419446123431\\
0.1	0.0176333645430056\\
};
\addlegendentry{$\|e_h\|_{L^2(H^1)}$}

\addplot [color=gray, dashed]
  table[row sep=crcr]{%
1.25	0.15\\
1	0.12\\
0.75	0.09\\
0.5	0.06\\
0.3	0.036\\
0.2	0.024\\
0.15	0.018\\
0.1	0.012\\
};
\addlegendentry{$\mathcal{O}(TOL)$}

\end{axis}
\end{tikzpicture}%
    \end{subfigure}
    \hfill
    \begin{subfigure}[t]{0.48\linewidth}
        \centering
%
\begin{tikzpicture}

\begin{axis}[%
width=0.8\linewidth,
height=5cm,
at={(0cm,0cm)},
scale only axis,
xmin=0,
xmax=1,
xlabel style={font=\color{white!15!black}},
xlabel={$t$},
ymode=log,
ymin=200,
ymax=5000,
yminorticks=true,
ylabel style={font=\color{white!15!black}},
ylabel={degrees of freedom},
axis background/.style={fill=white}
]
\addplot[const plot, color=black, line width=1.0pt, mark=o, mark options={solid, black}, forget plot] table[row sep=crcr] {%
0	3059\\
0.03125	3059\\
0.09375	2548\\
0.21875	2201\\
0.34375	1889\\
0.40625	1658\\
0.46875	1530\\
0.53125	1356\\
0.59375	1139\\
0.71875	941\\
0.78125	730\\
0.84375	560\\
0.90625	448\\
1	363\\
};
\end{axis}
\end{tikzpicture}%
    \end{subfigure}
    \caption{(left) $L^\infty(L^2)$- and $L^2(H^1)$-errors for a set of tolerances. (right) Number of nodes over time for an exponentially decaying solution ($\TOL = 1$).}
    \label{fig:ex1}
\end{figure}
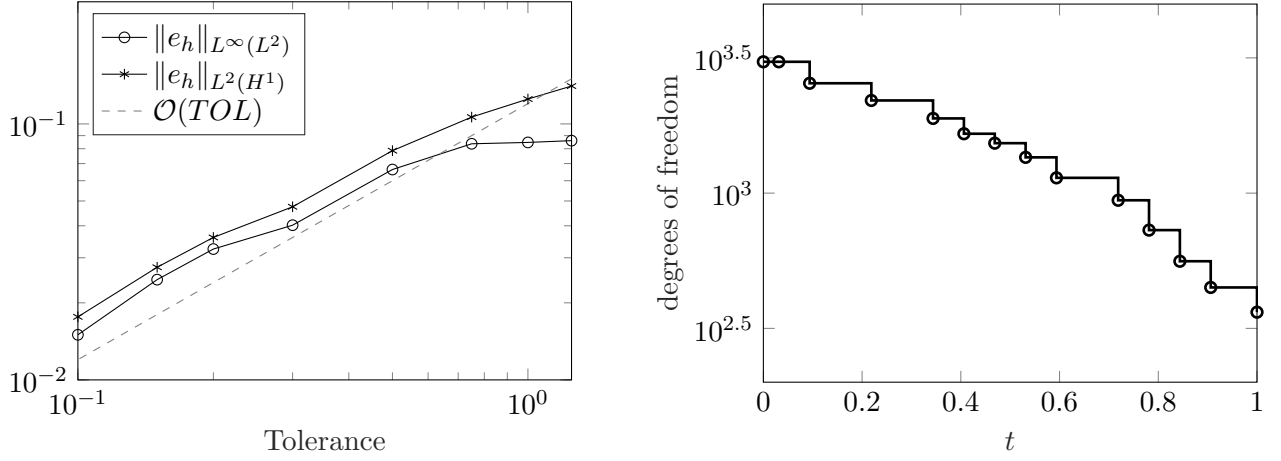


For comparability we computed the $L^2$-errors at the discrete timesteps for a non-adaptive uniform setting, with $\tau = 0.1$ and an initial mesh with approximately $2000$ almost uniformly distributed vertices, and the proposed adaptive routine with $\TOL = 0.7$.
In Figure~\ref{fig:ex1_adap_vs_nonadap} the $L^2$-errors over time corresponding to the right axis, are marked with triangles and the degrees of freedom (dof) over time corresponding to the left axis, are marked with circles. In black the adaptive results are shown, and in grey the non-adaptive results. The markers are set only for discrete timesteps at which the adaptive or uniform routine determined a discrete solution.
Note that similar to Figure~\ref{fig:ex1} (right) the exponential decay in dofs is visible for the adaptive version.

Initially, the adaptive routine demands more dof than the uniform discretisation; however, after $t =0.6$, the number of dof is reduced below that of the uniform mesh. 
Although, the adaptive routine exhibits a larger pointwise $L^2$-error for $t>1$, its overall $L^\infty(L^2)$-error is smaller than the uniform approach, shown by the dashed lines in Figure~\ref{fig:ex1_adap_vs_nonadap}.
This behaviour is consistent with the theory, as Theorem~\ref{thm:upper_lower} guarantees control of the $L^\infty(L^2)$-error rather than the pointwise $L^2$-error.
In particular, once the error is sufficiently small during the initial phase, adaptivity allows us to strongly reduce the dof while retaining good control of the $L^\infty(L^2)$-error. Note that this not only applies to spatial refinements but also the temporal time-step size, which is initially slightly smaller than the uniform approach, but later substantially larger without an error increase with respect to the error notion of Theorem~\ref{thm:upper_lower}. Note that similar behaviour is observed for the $L^2(H^1)$-norm.
\begin{figure}
    \centering
		\includegraphics[width=1\textwidth]{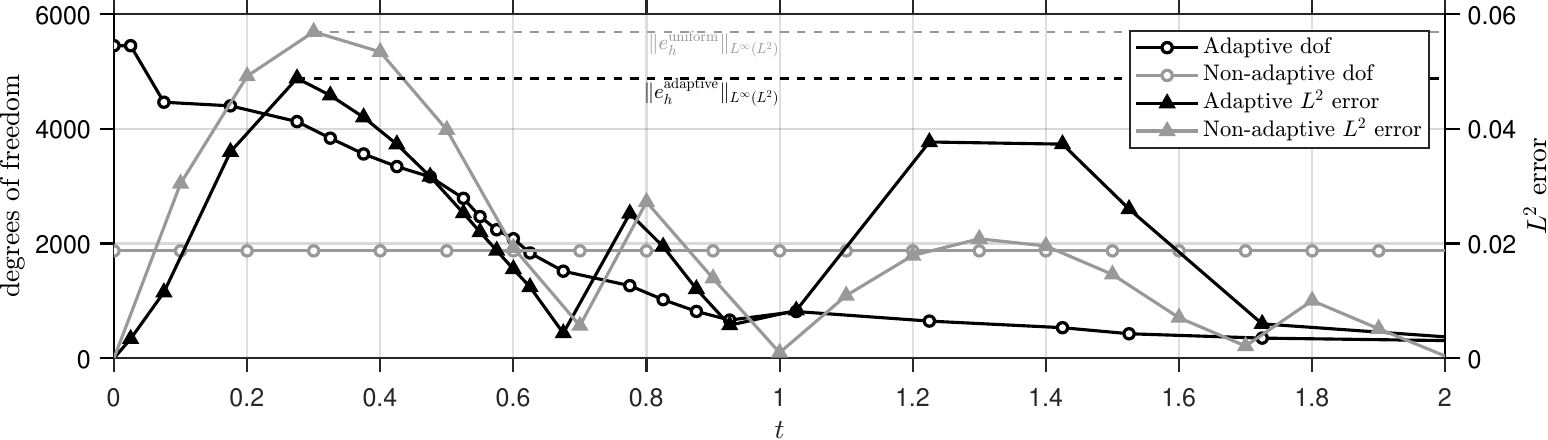}
        \caption{Double axis plot, with comparison of $L^2$-error (right $y$-axis) and degrees of freedom (left $y$-axis) over time for a uniform and adaptive solution of the bouncing ellipsoid with manufactured solution. The markers in general indicate for which discrete timesteps the uniform (in grey) and adaptive (in black) approximation, where computed. Additionally 
        the type of marker indicates that the corresponding $y$-axis is the $L^2$-error for triangles $\blacktriangle$ and respectively the degrees of freedom for circles $\circ$. The dashed lines indicate the global $L^\infty(L^2)$-error for the adaptive (in black) and uniform (in grey) approximation.}
        \label{fig:ex1_adap_vs_nonadap}
\end{figure}

\subsection{Dumbell with purely normal motion}
\label{sec:dumbell}
We analyse the experiment by Elliott and Styles \cite[Section~2]{elliott_styles_2012} which shows that due to the evolution of the surface nodes will start gathering along a strip. We highlight that for adaptive approaches, the coarsening allows us to limit the gathering by the choice of the initial baseline/coarsest mesh.

We define the moving domain by the level-set function
\[
d(x,y,z,t)
= x^2 + y^2
+ a(t)^2\, G\!\left(\frac{z^2}{L(t)^2}\right)
- a(t)^2,
\]
where
\[
G(s)=80\,s\left(s-\frac32\right),\qquad
a(t)=0.1+0.05\sin(2\pi t),\qquad
L(t)=1+0.2\sin(4\pi t).
\]
Now assuming that we have a purely normal motion as in \cite[Equation~2.1]{elliott_styles_2012}, the nodes of the dumbell move in a peculiar way. One observed \cite[Figure~1]{elliott_styles_2012} that such an example can lead to strongly non-uniform meshes. Adaptive strategies allow to weaken this effect.

We solve the PDE \eqref{eq:heat_strong} with $u_0 = 0$ and $f = 1$, i.e. adaptivity is driven by the geometry induced heat flow. For the non-adaptive setting we observe on the right side of Figure~\ref{fig:dumbell_adaptivity} strong gathering as in \cite[Figure~1]{elliott_styles_2012}.  
On the other hand adaptivity allows us to start with a very coarse base mesh, which allows to reduce the gathering effect by coarsening. It should be noted however that this effect is limited, as the base mesh has to be fine enough to resolve the geometry sufficiently well. But the coarsening yields a tool to limit the number of gathered nodes, by the number of initial nodes (as we can expect, that the error indicators, for elements in these dense areas, are small and marked for coarsening).
On the right side of Figure~\ref{fig:dumbell_adaptivity} the surface was initialized uniformly (based on DistMesh \cite{distmesh}) with $3024$ vertices. On the left side of Figure~\ref{fig:dumbell_adaptivity} the base mesh was given by $694$ nodes, due to adaptivity the adaptive solver with $\TOL = 2$ yields meshes with different number of nodes. At most $3908$ nodes were needed to resolve the problem with the provided tolerance. We can observe that the gathering effect is strongly reduced by comparing the non-adaptive and adaptive solutions at each timestep, however it is not fully avoidable. We want to remark however smaller tolerances would hide the gathering of the initial mesh even further and the effect can be expected to be neglectable asymptotically.
For better illustration of the mesh the final line in Figure~\ref{fig:dumbell_adaptivity}, i.e.\ the surface at time $t = 1$, is viewed from a slightly rotated position.

\begin{figure}[htbp!]
    \centering
    \includegraphics[width=0.9\textwidth]{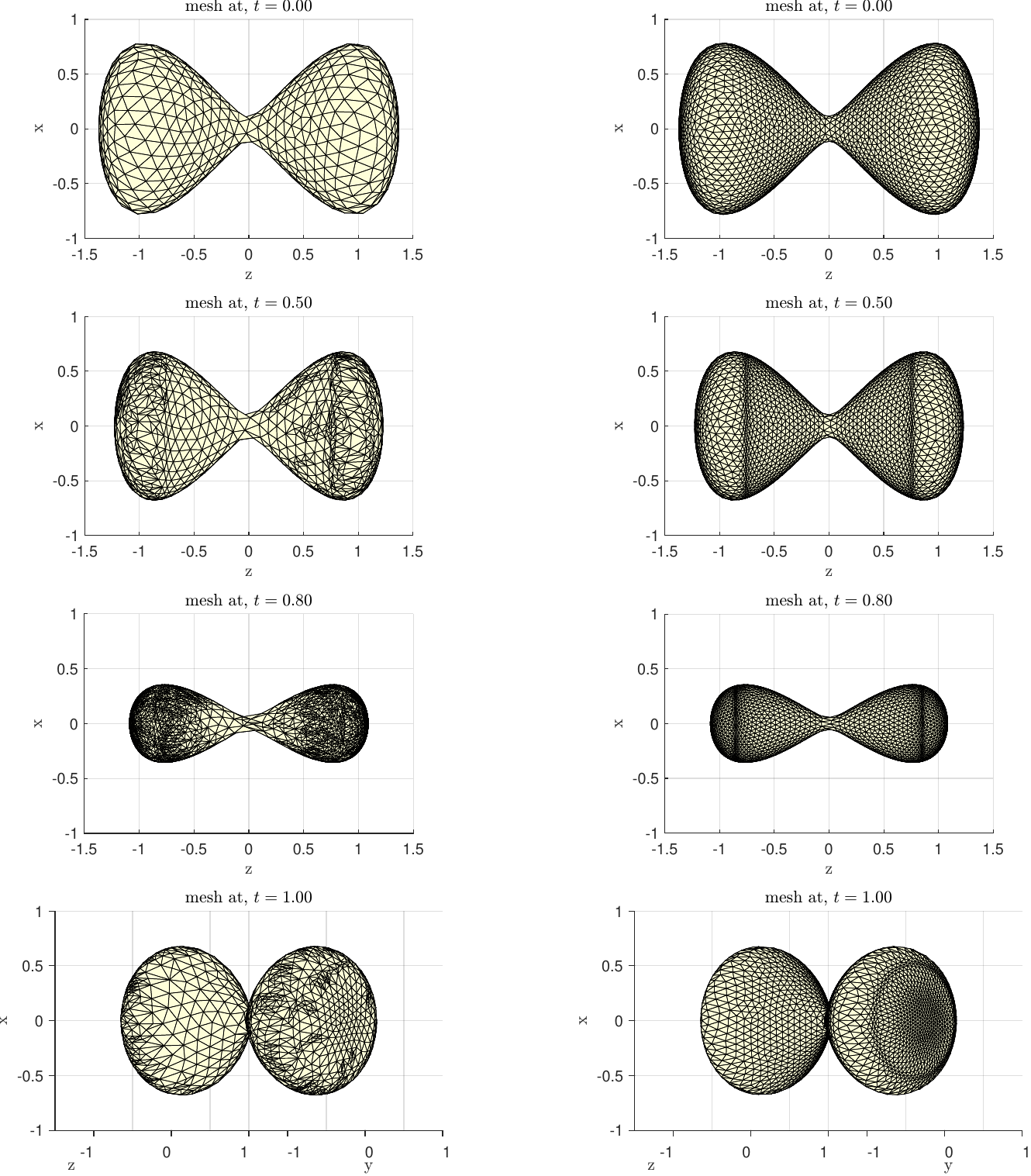}
    \caption{Dumbell experiment with (left) and without (right) adaptivity} 
    \label{fig:dumbell_adaptivity}
\end{figure}

We note that for the dumbell experiment, the mesh-transfer indicator $\eta_\textnormal{trans}$ exceeded its allocated budget, by at most a factor of $2$, in regions where the rapid evolution of the surface induced strong heat flow (in particular around $t =0.75$), which resulted in many refinements. According to the adaptive strategy outlined in the algorithm in the beginning of Section~\ref{sec:numerical_experiments}, this would trigger additional refinements of the previous timestep to reduce the transfer error. Since the present experiment is primarily intended to illustrate the improvement in mesh quality, and the budget is only moderately exceeded, we omit the mesh-correction step in this experiment.

\subsection{Reliability and Efficiency}
The next experiment heuristically analysis the reliability and efficiency of the estimator derived in \eqref{eq:indicator - full}. We determine the error indicator and the exact error for the non-adaptive approximation of Section~\ref{sec:ex1}, i.e.\ the bouncing ellipsoid with manufactured solution. Given a set of time-step sizes and initial meshes we observe the reliability and efficiency of the indicators in Figure~\ref{fig:reliability_eff}. As expected the estimator is only optimal with respect to the $L^2(H^1)$-error. Note that the mesh-width $h$ of Figure~\ref{fig:reliability_eff} is the maximal triangle diameter at time $t=0$ of the almost-uniform mesh generated using \cite{distmesh}.


\begin{figure}[htbp]
    \centering
		\resizebox{1\linewidth}{0.55\linewidth}{
%
%
\begin{tikzpicture}
\begin{axis}[%
width=8cm,
height=10cm,
at={(0cm,0cm)},
scale only axis,
xmode=log,
xmin=0.0347065388852302,
xmax=1.633,
xminorticks=true,
xlabel style={font=\color{white!15!black}},
xlabel={$h$},
ymode=log,
ymin=0.001,
ymax=1000,
yminorticks=true,
axis background/.style={fill=white},
title={\large{Estimator and $L^\infty(L^2)$  error for different $\tau$}},
legend style={at={(0.03,0.97)}, anchor=north west, legend cell align=left, align=left, draw=white!15!black,nodes={scale=0.75, transform shape}},
legend columns=2,
]
\addplot [color=white!40!black, mark=square, mark options={solid, white!40!black}]
  table[row sep=crcr]{%
1.633	0.0867254391772709\\
1.0108989431474	0.0599684742867465\\
0.541588514995021	0.0475916318684795\\
0.275912115651395	0.0839220523057472\\
0.138617572085068	0.0945512645755848\\
0.0693921836466527	0.0973168485050612\\
0.0347065388852302	0.0980156153328137\\
};
\addlegendentry{$L^\infty(L^2)$ error $(\tau = 1)$}

\addplot [color=white!70!black, mark=square, mark options={solid, white!70!black}]
  table[row sep=crcr]{%
1.633	14.1945786298292\\
1.0108989431474	10.5077250344713\\
0.541588514995021	9.26924966348745\\
0.275912115651395	7.92455603239824\\
0.138617572085068	7.0664540117227\\
0.0693921836466527	6.67740900168324\\
0.0347065388852302	6.50078628521261\\
};
\addlegendentry{estimator $\eta$ $(\tau = 1)$}

\addplot [color=white!40!black, mark=o, mark options={solid, white!40!black}]
  table[row sep=crcr]{%
1.633	0.36263200570435\\
1.0108989431474	0.349471777214253\\
0.541588514995021	0.150989524031221\\
0.275912115651395	0.0799009306773504\\
0.138617572085068	0.0599641902405679\\
0.0693921836466527	0.0548770576111311\\
0.0347065388852302	0.053601111536652\\
};
\addlegendentry{$L^\infty(L^2)$ error $(\tau = 0.1)$}

\addplot [color=white!70!black, mark=o, mark options={solid, white!70!black}]
  table[row sep=crcr]{%
1.633	6.88028586573875\\
1.0108989431474	4.83078435226165\\
0.541588514995021	3.63154328054793\\
0.275912115651395	2.26144436795475\\
0.138617572085068	1.5216948165132\\
0.0693921836466527	1.16197722922784\\
0.0347065388852302	0.98796305289336\\
};
\addlegendentry{estimator $\eta$ $(\tau = 0.1)$}

\addplot [color=white!40!black, mark=x, mark options={solid, white!40!black}]
  table[row sep=crcr]{%
1.633	0.37182659559317\\
1.0108989431474	0.343580759319753\\
0.541588514995021	0.118374638146261\\
0.275912115651395	0.0368264701625841\\
0.138617572085068	0.0136198097514683\\
0.0693921836466527	0.00768639836565636\\
0.0347065388852302	0.00622039609443846\\
};
\addlegendentry{$L^\infty(L^2)$ error $(\tau = 0.01)$}

\addplot [color=white!70!black, mark=x, mark options={solid, white!70!black}]
  table[row sep=crcr]{%
1.633	6.53078567968271\\
1.0108989431474	4.50128622099598\\
0.541588514995021	3.11588502159964\\
0.275912115651395	1.63049395566111\\
0.138617572085068	0.843222176852563\\
0.0693921836466527	0.458728414529415\\
0.0347065388852302	0.271681076128634\\
};
\addlegendentry{estimator $\eta$ $(\tau = 0.01)$}

\addplot [color=black, dotted]
  table[row sep=crcr]{%
1.633	16.33\\
1.0108989431474	10.108989431474\\
0.541588514995021	5.41588514995021\\
0.275912115651395	2.75912115651395\\
0.138617572085068	1.38617572085068\\
0.0693921836466527	0.693921836466527\\
0.0347065388852302	0.347065388852302\\
};
\addlegendentry{$\mathcal{O} (h)$}

\addplot [color=black, dashed]
  table[row sep=crcr]{%
1.633	8.000067\\
1.0108989431474	3.06575001976957\\
0.541588514995021	0.879954358723537\\
0.275912115651395	0.228382486689686\\
0.138617572085068	0.0576444938722774\\
0.0693921836466527	0.0144458254537523\\
0.0347065388852302	0.00361363152417598\\
};
\addlegendentry{$\mathcal{O} (h^2)$}

\end{axis}

\begin{axis}[%
width=8cm,
height=10cm,
at={(10cm,0cm)},
scale only axis,
xmode=log,
xmin=0.0347065388852302,
xmax=1.633,
xminorticks=true,
xlabel style={font=\color{white!15!black}},
xlabel={$h$},
ymode=log,
ymin=0.01,
ymax=100,
yminorticks=true,
axis background/.style={fill=white},
title={\large{Estimator and $L^\infty(L^2)$  error for different $\tau$}},
legend style={at={(0.03,0.97)}, anchor=north west, legend cell align=left, align=left, draw=white!15!black,nodes={scale=0.75, transform shape}},
legend columns=2,
]

\addplot [color=white!30!black, mark=square, mark options={solid, white!30!black}]
  table[row sep=crcr]{%
1.633	0.188909224859433\\
1.0108989431474	0.20411869752478\\
0.541588514995021	0.172478393180128\\
0.275912115651395	0.17922961830671\\
0.138617572085068	0.182661069537074\\
0.0693921836466527	0.183632813077133\\
0.0347065388852302	0.183882278589303\\
};
\addlegendentry{$L^2(H^1)$ error $(\tau = 1)$}

\addplot [color=white!70!black, mark=square, mark options={solid, white!70!black}]
  table[row sep=crcr]{%
1.633	14.1945786298292\\
1.0108989431474	10.5077250344713\\
0.541588514995021	9.26924966348745\\
0.275912115651395	7.92455603239824\\
0.138617572085068	7.0664540117227\\
0.0693921836466527	6.67740900168324\\
0.0347065388852302	6.50078628521261\\
};
\addlegendentry{estimator $\eta$ $(\tau = 1)$}

\addplot [color=white!30!black, mark=o, mark options={solid, white!30!black}]
  table[row sep=crcr]{%
1.633	0.541325357755142\\
1.0108989431474	0.610118964342365\\
0.541588514995021	0.33969250539036\\
0.275912115651395	0.188749525974572\\
0.138617572085068	0.119880181262861\\
0.0693921836466527	0.0946687846028215\\
0.0347065388852302	0.0871909756827554\\
};
\addlegendentry{$L^2(H^1)$ error $(\tau = 0.1)$}

\addplot [color=white!70!black, mark=o, mark options={solid, white!70!black}]
  table[row sep=crcr]{%
1.633	6.88028586573875\\
1.0108989431474	4.83078435226165\\
0.541588514995021	3.63154328054793\\
0.275912115651395	2.26144436795475\\
0.138617572085068	1.5216948165132\\
0.0693921836466527	1.16197722922784\\
0.0347065388852302	0.98796305289336\\
};
\addlegendentry{estimator $\eta$ $(\tau = 0.1)$}

\addplot [color=white!30!black, mark=x, mark options={solid, white!30!black}]
  table[row sep=crcr]{%
1.633	0.567929218254469\\
1.0108989431474	0.629435918547889\\
0.541588514995021	0.335128806945937\\
0.275912115651395	0.170284719991626\\
0.138617572085068	0.0858759399203111\\
0.0693921836466527	0.0437548106838108\\
0.0347065388852302	0.0233310191510394\\
};
\addlegendentry{$L^2(H^1)$ error $(\tau = 0.01)$}

\addplot [color=white!70!black, mark=x, mark options={solid, white!70!black}]
  table[row sep=crcr]{%
1.633	6.53078567968271\\
1.0108989431474	4.50128622099598\\
0.541588514995021	3.11588502159964\\
0.275912115651395	1.63049395566111\\
0.138617572085068	0.843222176852563\\
0.0693921836466527	0.458728414529415\\
0.0347065388852302	0.271681076128634\\
};
\addlegendentry{estimator $\eta$ $(\tau = 0.01)$}

\addplot [color=black, dotted]
  table[row sep=crcr]{%
1.633	1.633\\
1.0108989431474	1.0108989431474\\
0.541588514995021	0.541588514995021\\
0.275912115651395	0.275912115651395\\
0.138617572085068	0.138617572085068\\
0.0693921836466527	0.0693921836466527\\
0.0347065388852302	0.0347065388852302\\
};
\addlegendentry{$\mathcal{O} (h)$}

\end{axis}

\end{tikzpicture}%
}
        \caption{Efficiency and reliability analysis for the bouncing ellipsoid problem with $u = xy \exp(t)$.}
        \label{fig:reliability_eff}
\end{figure}
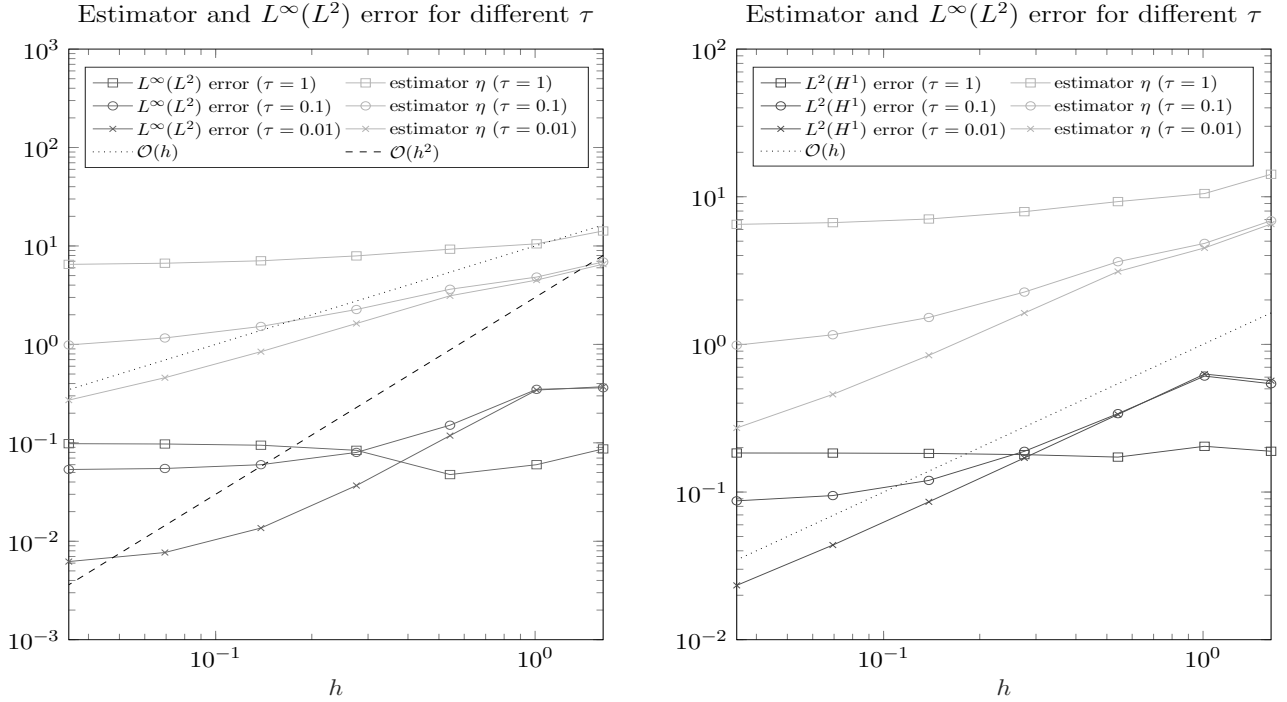



\subsection{Moving Peak: Coarsening and refinement}
We illustrate the coarsening and refinement on evolving surfaces by analysing a moving peak on a bouncing ellipsoid, following the non-evolving example of \cite{KL25}.
Based on the exact solution 
\begin{equation}
\label{eq:movingpeak}
u(x,y,z,t) = \Big(1-\exp\big(-200(t-0.5)^2\big)\Big)\exp\Big(-25\big((z-\cos(\pi t/T))^2 + (y-\sin(\pi t/T))^2 + x^2\big)\Big),
\end{equation}
which resembles a moving heat source travelling along the $yz$-plane, while vanishing briefly at $t= 0.5$. Again we computed the corresponding right-hand side based on \eqref{eq:heat_strong}. The parameter $T = 2$ is used to control that the peak moves a quarter revolution for $t \in [0,1]$. For both upcoming plots we chose $\TOL = 1$. For improved visualization of the adaptive meshes, we additionally scaled the temporal tolerances by $\theta_\tau = 0.3$ (i.e.\ in step (4) of the schematic algorithm of Section~\ref{sec:numerical_experiments}).
The movement of the surface, the refining, and the coarsening can be seen in Figure~\ref{fig:MeshesPeak}.
 It should be noted that the gradient of the peak is smaller when the surface expands, thus requiring fewer nodes, whereas the gradient increases when the surface shrinks, resulting in stronger refinement. 
 Around the midpoint $t =0.5$, where the peak temporarily disappears, the mesh is strongly coarsened. However, due to the dependence of information on two consecutive time-levels in the indicator, the comparatively large timesteps taking in this phase (see Figure~\ref{fig:ex5_adap_vs_nonadap}), and the restriction to single-level coarsening in our implementation, there is still a set of refined elements.
The dof over time, marked by black triangles $\blacktriangle$ in Figure~\ref{fig:ex5_adap_vs_nonadap} show a strong decrease in the dofs around the midpoint. However it is slightly delayed, mainly due to the restriction to single-level coarsening. Allowing for multiple-level coarsening is expected to further enhance this effect.

\begin{figure}[htbp]
    \centering
    \includegraphics[width=1\linewidth]{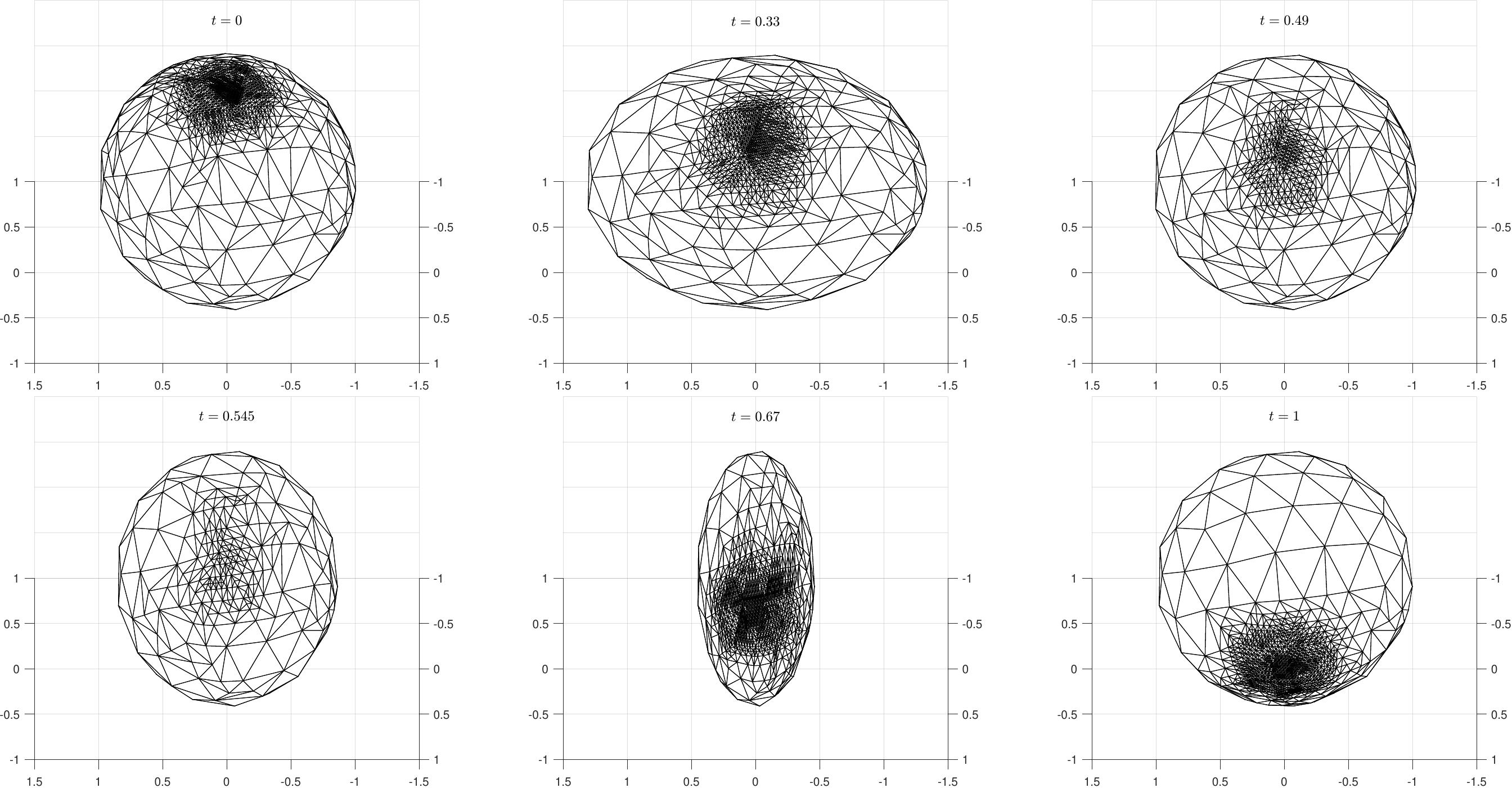} 
    \caption{Adaptively obtained meshes at different timesteps for the moving peak experiment ($\TOL = 1$).}
	\label{fig:MeshesPeak}
\end{figure}

We also computed the local in time $H^1$-error for an adaptive and a uniform discretization with roughly 2000 dofs and $\tau = 0.05$. The results are given over time in Figure~\ref{fig:ex5_adap_vs_nonadap}. We observe that adaptivity is highly beneficial for this example and the $H^1$-error is controlled and kept nearly constant, while managing, domain change, strong source-term change and geometric errors. 

\begin{figure}[htbp]
    \centering
		\includegraphics[width=1\textwidth]{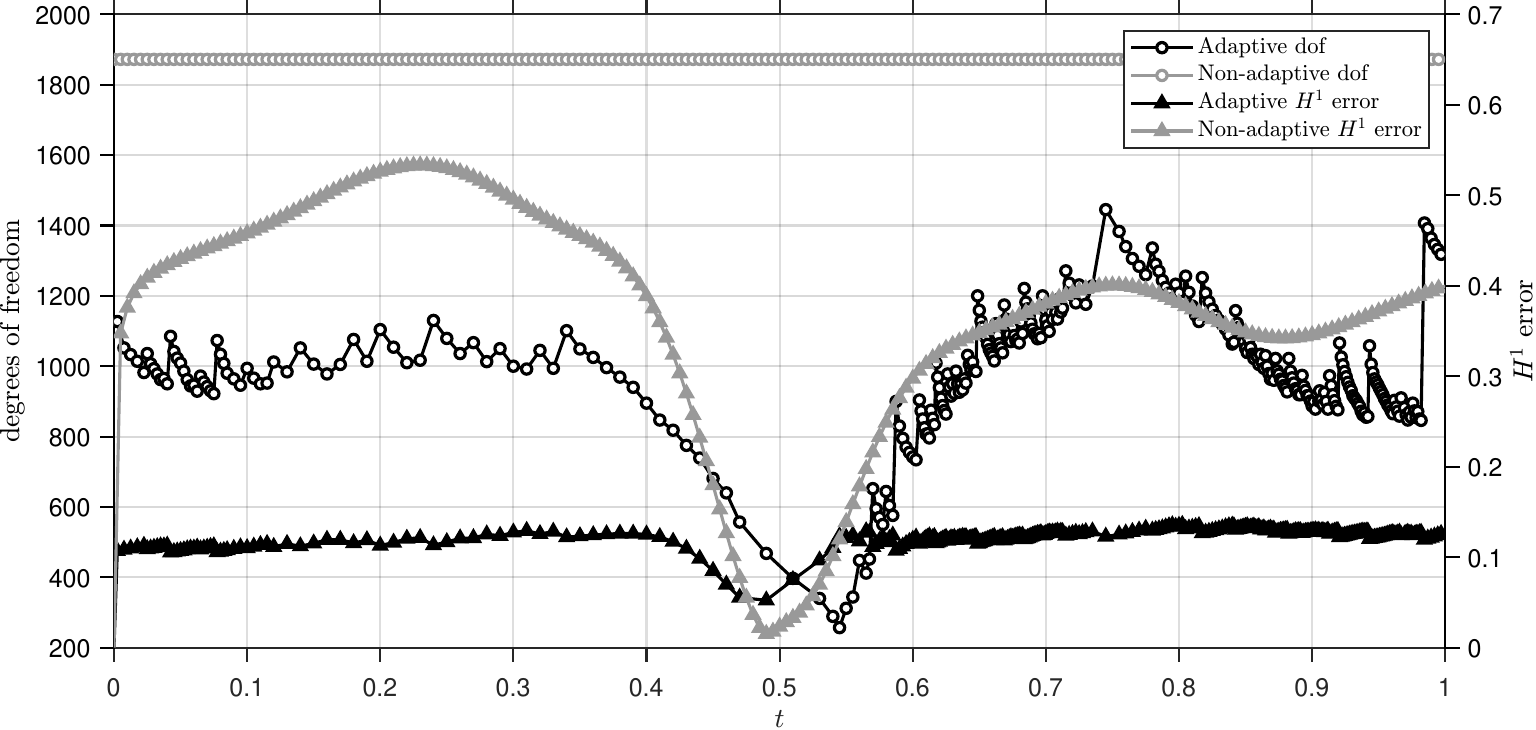}
        \caption{Double axis plot, with comparison of $H^1$-error (right $y$-axis) and degrees of freedom (left $y$-axis) over time for a uniform and adaptive solution ($\TOL = 1$) of the moving peak example \eqref{eq:movingpeak} on the bouncing ellipsoid. The markers in general indicate which discrete timesteps the uniform (in grey) and adaptive (in black) approximation, where computed. Additionally the type of marker indicates that the corresponding $y$-axis is the $L^2$-error for triangles $\blacktriangle$ and respectively the degrees of freedom for circles $\circ$.}
        \label{fig:ex5_adap_vs_nonadap}
\end{figure}

\clearpage

\bibliographystyle{abbrvnat}
\bibliography{aesfem_bib}


\end{document}